\documentclass[11pt]{amsart}
\usepackage[T1]{fontenc}
\usepackage{mathptmx}

\usepackage{comment}
\usepackage[colorlinks=true,linkcolor=magenta,citecolor=blue]{hyperref}     
\usepackage[capitalise]{cleveref}   
\numberwithin{equation}{section}
\usepackage{mathtools}
\usepackage{amsmath, amssymb,amsthm}
\usepackage{array}
\usepackage{graphicx}
\usepackage{float}
\usepackage{caption,subcaption}
\usepackage{tikz}
\usepackage{tikz-cd}
\usetikzlibrary{matrix,arrows,decorations.pathmorphing}
\usepackage[mathscr]{euscript}
\usepackage{MnSymbol}

\def\A{\mathbb{A}}
\def\Dbb{\mathbb{D}}
\def\R{\mathbb{R}}
\def\C{\mathbb{C}}
\def\Q{\mathbb{Q}}
\def\Z{\mathbb{Z}}
\def\N{\mathbb{N}}
\def\P{\mathbb{P}}
\def\K{\mathbb{K}}

\def\O{\mathcal{O}}
\def\X{\mathfrak{X}}
\def\Y{\mathfrak{Y}}

\def\B{\mathfrak{B}}
\def\L{\mathfrak{L}}
\def\Uscr{\mathscr{U}}
\def\Fscr{\mathscr{F}}
\def\Gscr{\mathscr{G}}
\def\Cscr{\mathscr{C}}
\def\Ascr{\mathscr{A}}
\def\Vscr{\mathscr{V}}
\def\Dscr{\mathscr{D}}
\def\Mscr{\mathscr{M}}
\def\Vscr{\mathscr{V}}
\def\Zscr{\mathscr{Z}}

\def\E{\mathscr{E}}
\def\D{\mathscr{D}}

\def\M{\mathscr{M}}

\newtheorem{thm}{Theorem}[section]
\newtheorem{lem}[thm]{Lemma}
\newtheorem{ex}[thm]{Example}
\newtheorem{rem}[thm]{Remark}

\newtheorem{prop}[thm]{Proposition}
\newtheorem{claim}[thm]{Claim}

\newcommand\Mbar {\overline{\Mscr}}
\newcommand\Qbar {\overline{\Q}}
\newcommand\Dbar {\overline{\Dscr}}
\newcommand\Lbar {\overline{\L}}

\newcommand\an{\mathrm{an}}
\newcommand\supp{\mathrm{supp}}
\newcommand\Jac{\mathrm{Jac}}
\newcommand\vol{\operatorname{vol}}
\newcommand\gal{\operatorname{Gal}}

\newcommand\numberthis{\addtocounter{equation}{1}\tag{\theequation}}
\DeclareMathOperator{\ddc}{dd^c}
\DeclareMathOperator{\mdeg}{mdeg}

\title{Arithmetic properties of families of plane polynomial automorphisms}
\author{Yugang ZHANG}
\address{Laboratoire de Math\'ematiques d’Orsay, Université Paris-Saclay, 307 Rue Michel Magat, 91405 Orsay,
France}
\email{\href{yugang.zhang@universite-paris-saclay.fr}{yugang.zhang@universite-paris-saclay.fr}}
\date{\today}
\begin{document}

\begin{abstract}
    Given an algebraic family $f\colon\Lambda\times \A^2 \to \Lambda\times \A^2$ of plane polynomial automorphisms of H\'enon type parameterized by a quasi-projective curve, defined over a number field $\K,$ we investigate certain arithmetic properties of periodic points contained in a family of subvarieties $X \subset \Lambda \times \A^2 \twoheadrightarrow \Lambda$. 
    
    First, consider $X$ as a curve. We prove that the set of parameters $t\in\Lambda(\Qbar)$, such that $X_t$ is periodic, has bounded height. This generalizes a result of Patrick Ingram. Moreover, if $X$ is non-periodic, then under some mild conditions --- such as when the family is dissipative --- we show that there are, in fact, only finitely many periodic parameters. This extends a result of Charles Favre and Romain Dujardin. 

    Second, let $X$ be a family of curves. Assuming $X$ is non-degenerate, we establish a uniform bound on the number of periodic points in each curve $X_t$, $t\in \Lambda(\Qbar)$ and show that the set of these periodic points have bounded height in $\Lambda\times \A^2$ as well. We then examine in more detail the non-degeneracy property in the case of dissipative families of quadratic H\'enon maps.
\end{abstract}

\maketitle

\section{introduction}
\subsection{Marked points and relative Manin-Mumford}
\subsubsection{Background and motivation}
Let $f_t \colon \P^1 \to \P^1$ be a complex algebraic family of rational maps of degree $d \geq 2$, parameterized by a smooth complex quasi-projective variety $\Lambda$. By a \emph{marked point}, we mean a morphism $\sigma \colon \Lambda \to \P^1$. If the pair $(f_t, \sigma)$ is \emph{non-isotrivial} (i.e., $f_t$ and $\sigma$ do not depend on $t$, up to conjugating by a family of Möbius transformations and up to a base change), then by Montel's theorem, there are infinitely many $t$ such that $\sigma(t)$ is preperiodic for $f_t$ (see e.g., \cite[Lemma 2.3.]{DujardinFavre08}). In particular, for a non-isotrivial family of elliptic curves, the set of torsion points on the image of a section is infinite (see e.g., \cite{DemarcoANT,MZ10}).

In the case of higher relative dimension, Gao and Habegger \cite{GaoHabeggerManinMumford} recently proved the so-called relative Manin-Mumford conjecture, which can be stated as follows. Let $\pi\colon \Ascr\to \Lambda$ be a family of abelian varieties of relative dimension $g\geq 1,$ defined over an algebraically closed field of characteristic 0. Let $X$ be an irreducible subvariety of $\Ascr.$ Assume that $\Z X \coloneqq \bigcup_{N\in \Z}[N] X$ is Zariski dense in $\Ascr.$ If the torsion points in $X$ are Zariski dense, then $\dim X \geq g.$ In particular, if $\Ascr \to \Lambda$ is a family of simple abelian varieties and $X$ has dimension 1, then, in contrast to the case of relative dimension 1, there exist only finitely many torsion points on $X.$
Several earlier results on the relative Manin-Mumford conjecture can be found in e.g., \cite{CTZ23,DemarcoMavraki20,Habegger13,KuhneRelativeBogomolov,MZ15,Stoll}. The conjecture was inspired by Zhang's ICM talk \cite{ZhangICM} and was proposed by Pink \cite{Pink05} and Zannier \cite{Zannier12}.

\subsubsection{Dynamical settings and main results}
Our first goal in this work is to provide an analogue of relative Manin-Mumford in the context of families of dynamical systems. Let $f\colon \Lambda\times \A^2 \to \Lambda\times \A^2$ be a family of plane polynomial automorphisms parameterized by a smooth quasi-projective curve $\Lambda$, defined by $f(t,z)=(t,f_t(z))$. Assume both $f$ and $\Lambda$ are defined over a number field $\K.$ A \emph{marked point} is a morphism $\sigma\colon \Lambda \to \A^2$. Since we are interested in iterating the marked point by $f$, it is convenient to extend the term ``marked point'' to also refer to the graph morphism $\Lambda\to \Lambda\times \A^2$. When we say that a marked point has a periodic point, we mean that there exists $t\in\Lambda$ such that $\sigma(t)$ is periodic for $f_t.$ A marked point is said to be \emph{(globally) periodic} if there exists two distinct positive integers $n$ and $m$ such that  $f^n_t(\sigma(t))=f^m_t(\sigma(t))$ for all $t.$ In this work, we investigate the number of periodic points associated with a marked point.

A plane polynomial automorphism is called a \emph{generalized H\'enon map} if it is a finite composition of maps of the form $(x,y)\mapsto (y,p(y)-\delta x),$ where $p$ is a non-invertible monic polynomial and $\delta$ is a non-zero constant.  Such a map is said to be of \emph{H\'enon type} if it is conjugate, via a polynomial automorphism, to a generalized H\'non map. Friedland and Milnor \cite{FriedlandMilnor} showed that the only dynamically interesting plane polynomial automorphisms are of H\'enon type, in the sense that their first dynamical degree is grester than 1. Consequentely, we will focus on this class of polynomial automorphisms. Furthermore, since our results are independent of the conjugacy class, we may restrict our attention to \emph{regular plane polynomial automorphisms}, that is, those that are affine conjugate to generalized H\'enon maps.

Recall that any polynomial automorphism has a constant Jacobian. Therefore, the Jacobian function $t\in \Lambda(\C) \to \Jac(f_t)\in \C$ is well-defined and regular. This function is either constant or surjective onto $\C$ up to finitely many values. Our first result addresses the case where the Jacobian function is constant.

\begin{thm}\label{MainMarkedPoint}
    Let $f_t$ be an algebraic family of polynomial automorphisms of H\'enon type of degree $d\geq 2$ parameterized by a smooth quasi-projective curve $\Lambda$, defined over a number field $\K.$ Let $\sigma\colon\Lambda \to \A^2$ be a non-globally periodic marked point defined over $\K.$ Fix an archimedean place $v$ of $\K$. Suppose the Jacobian $\Jac(f_t)$ is constant and $|\Jac(f_t)|_v\neq 1.$ Then there exist only finitely many $t\in \Lambda(\C)$ such that $\sigma(t)$ is periodic for $f_t$.
\end{thm}

In Theorem \ref{MainMarkedPoint}, if $f_t=g$ for some fixed $g$, then the statement reduces to a result of Dujardin and Favre (see \cite[Theorem A']{DFManinMumford2017}). In this case, they showed more generally that the set of algebraic points with small canonical heights is finite.
On the other hand, if $|\Jac(f_t)|_v= 1$, a marked point may have infinitely many periodic points. This phenomenon occurs when the polynomial automorphisms are reversible, a situation first noted in \cite[Proposition 7.1]{DFManinMumford2017}. See also \cite[Theorem D]{hsia2018heights}.

Let $t\in \Lambda$ be a parameter such that $\sigma(t)$ is  $f_t$-periodic of period $k$. Let $u(t)$ and $s(t)$ be the two eigenvalues of the differential of $f_t^k$ at $\sigma(t)$. Then we say that $\sigma(t)$ (or just $t$) is 
\begin{itemize}
    \item \emph{saddle} if $|u(t)| > 1 > |s(t)|$;
    \item \emph{semi-repelling} if $|u(t)|>1$ and $|s(t)|=1$;
    \item \emph{repelling} if $|u(t)|>1$ and $|s(t)|>1$;
    \item \emph{neutral} if $|u(t)|=|s(t)|>1$.
\end{itemize}
\begin{thm}\label{MainMarkedPoint2}
    Let $f_t$ be an algebraic family of polynomial automorphisms of H\'enon type of degree $d\geq 2$ parameterized by a smooth quasi-projective curve $\Lambda$, defined over a number field $\K.$ Let $\sigma\colon\Lambda \to \A^2$ be a non globally periodic marked point defined over $\K.$ Suppose $\Jac(f_t)$ is not persistently equal to a root of unity. If there exist infinitely many $t\in \Lambda(\C)$ such that $\sigma(t)$ is periodic, then all of them are neutral.
\end{thm}

As mentioned above, a marked point $\sigma$
may have infinitely many parameters $t$ such that $\sigma(t)$ is periodic when $|\Jac(f_t)|=1$. Nevertheless, we are able to show that the set of such parameters has bounded height.
\begin{thm}[cf.~Theorem~\ref{boundedheight}]\label{MainBoundedHeight}
    Let $f_t$ be an algebraic family of polynomial automorphisms of H\'enon type of degree $d\geq 2$ parameterized by a smooth quasi-projective curve $\Lambda$, defined over a number field $\K.$ Let $\sigma$ be a marked point. Then
    \begin{align*}
        \{t\in \Lambda(\Qbar) \mid \sigma(t)\ \text{is periodic for}\ f_t \}
    \end{align*}
    is a set of bounded height.
\end{thm}
The special case of Theorem~\ref{MainBoundedHeight} where $f_t$ is a H\'enon map of the form $(y,x+f_t(y))$ (in particular it has constant jacobian -1) and non-isotrivial, was proved by Ingram \cite[Theorem 1.3]{Ingram2011CanonicalHF}.

Hsia and Kawaguchi \cite[Theorem G]{hsia2018heights} investigated, on the other hand, ``unlikely intersection problems'' for families of H\'enon maps with constant Jacobian $\pm 1$, parameterized by the affine line $\A^1$. More precisely, given two marked points $\sigma_1, \sigma_2 \colon \A^1 \to \A^2$, each containing infinitely many periodic points (a situation that does not occur when the Jacobian is not on the unit circle, as shown by Theorem~\ref{MainMarkedPoint}), they studied the case when the set of $t$ such that $\sigma_1(t)$ and $\sigma_2(t)$ are both periodic is infinite.

\subsubsection{Transfer from the parameter space to the phase space and equidistribution}
We will prove two renormalization lemmas in Sect.~\ref{sectionrenormalization} --- one for saddle periodic points and another for semi-repelling periodic points. These lemmas are analogues of Tan lei's similarity theorem \cite{Tanlei} for repelling preperiodic points. Roughly speaking, they allows us to transfer information from the parameter space $\Lambda$ to the dynamical space $A^2$, by rescaling with an appropriate factor and iterating. This technique has become widely used in the study of holomorphic dynamical systems (see e.g., \cite{BuffEpstein,DFManinMumford2017,FavreGauthier2022,daoforcurves}). Unlike the case for repelling preperiodic points of rational maps, a map typically cannot be linearized at a saddle point, and the situation is even more intricate for semi-repelling periodic points. Our renormalization lemmas work for these latter cases within families. 
It seems natural to expect that Theorem~\ref{MainMarkedPoint2} should reach the same conclusion as Theorem~\ref{MainMarkedPoint}. However, our approach does not apply to neutral periodic points, as there is less known about their behavior in general.

Let us briefly outline the strategy for proving Theorems~\ref{MainMarkedPoint} and~\ref{MainMarkedPoint2}. Fix an archimedean place $v$ of $\K$ and let $G_f^\pm\colon\Lambda\times \A^2(\C)\to \R_{\geq 0}$ be the \emph{fibered forward/backward Green functions} (associated with $f_t$), defined by
\begin{align}\label{Greenfunction+-}
    G_f^\pm(t,z)=G_{f_t}^\pm(z)=G_t^\pm(z)\coloneqq\lim_{n\to +\infty}\frac{1}{d^n}\log^+\lVert f_t^{\pm n}(z)\rVert,
\end{align}
which are non-negative, continuous, and plurisubharmonic (see \cite{BS1,GVhenon,Hubbard86,SibonyPanorama}). Hence, to a marked point $\sigma:\Lambda\to \Lambda\times \A^2$, we can associate the \emph{forward/backward Green measure} (of $f_t$) on $\Lambda(\C)$, given by
\begin{align}\label{Greenmeasuremarkedpoint}
\mu_{f,\sigma}^\pm \coloneqq \sigma^*\ddc G_f^\pm.    
\end{align}
To prove Theorem~\ref{MainMarkedPoint}, we assume, for the sake of contradiction, that there exist infinitely many $t$ such that $\sigma(t)$ is periodic. This will imply, by the equidistribution Theorem~\ref{equidistributionpoints} (see also Sect.~\ref{introequi} below), that the two Green measures $\sigma_{f,\sigma}^+$ and $\sigma_{f,\sigma}^-$ are proportional. Then, applying our renormalization Lemmas~\ref{renormalization semi} and~\ref{renormalization saddle}, we will derive a contradiction. For Theorem~\ref{MainMarkedPoint2}, we additionally rely on a recent result\footnote{I would like to thank Dujardin for informing me of this result, see \cite[Proposition 2.2]{CantatDujardin}} by Cantat and Dujardin concerning the rigidity of the Julia set of regular plane polynomial automorphisms, see Lemma~\ref{RigidityDujardinCantat}.

\subsection{Families of curves and uniform dynamical Bogomolov}

A marked point $\sigma:\Lambda \to\Lambda\times \A^2$ can be viewed as a horizontal subvariety (its image) in the total space $\Lambda\times \A^2$ of relative dimension zero. We now turn to the study of horizontal subvarieties of relative dimension one. We say that a subvariety $\Cscr\subset\Lambda\times\A^2$ defined over a number field $\K$ is \emph{a family of curves parameterized by} $\Lambda$ if the projection $\Cscr \to \Lambda$ is smooth and each geometric fiber $\Cscr_t,$ $t\in \Lambda(\Qbar)$, is a integral curve. 

Dujardin and Favre \cite[Theorem A'']{DFManinMumford2017} proved that, given a regular plane polynomial automorphism defined over a number field $\K$ whose Jacobian does not lie on the unit cycle at some archimedean place, the number of points of small canonical height on an algebraic curve depends only on the degree of the curve. We extend this result by examining how this dependence varies with the map itself. 

\subsubsection{From arithmetic properties to pluripotential theory}
We first translate this arithmetic uniformity property into the non-vanishing property of some Green measure associated with the given family of curves. Specifically, Fix an archimedean place $v$ of $\K$ so that we have $G_f^\pm$ \eqref{Greenfunction+-}, then define
\begin{align}\label{Greenfunctionmax}
G_f(t,z)=G_{f_t}(z)=G_t(z) \coloneqq \max\{ G^+_{f}(t,z),G^+_{f}(t,z)\}.
\end{align}
In the same spirit of the work \cite{gauthier2020geometric} of Gauthier and Vigny on families of endomorphisms of projective spaces, we define the Green measure associated with $\Cscr$ as
\begin{align}\label{green measure}
    \mu_{f,\Cscr}\coloneqq (\ddc G_f)^2\wedge [\Cscr].
\end{align}
Following Yuan and Zhang \cite{YZadelic}, we say that $\Cscr$ is \emph{non-degenerate} if $\mu_{f,\Cscr}$ is non-vanishing. We can show the following.
\begin{thm}\label{maincurve}
    Let $f_t$ be an algebraic family of regular plane polynomial automorphisms of degree $d\geq 2$ parameterized by a smooth quasi-projective curve $\Lambda$ defined over a number field $\K.$ Fix an archimedean place $v$ of $\K,$ so that we have an embedding $\K \hookrightarrow \C.$ Let $\Cscr$ be a non-degenerate family of curves. 
    
    Then there exist a positive constant $\varepsilon$ and a positive integer $N$ such that for all but finitely many $t\in \Lambda(\Qbar)$, the set $\{z\in \Cscr_t(\Qbar)\ |\ \hat{h}_{f_t}(z) \leq \varepsilon \}$ has at most $N$ points. In particular, for all but finitely many $t\in \Lambda(\C)$, there are at most $N$ periodic points on $\Cscr_t(\C).$
\end{thm}
Note that $\varepsilon$ and $N$ depend on the families $f$ and $\Cscr.$ We refer to Sect. \ref{SectHeightInequalities} for the definition of the canonical height function $\hat{h}_{f_t}$ of a regular plan polynomial automorphism. To prove Theorem~\ref{maincurve}, we again proceed by contradiction. We suppose the conclusion is not satisfied and then apply the equidistribution Theorem~\ref{equidistribution curves} to contradict Proposition~\ref{positivityI_f}.

The set of periodic points contained in a non-degenerate surface $\Cscr$ may still be infinite, but it is a set of bounded height as well. This is the following theorem.
\begin{thm}\label{MainBoundedHeightSurface}
    Under the non-degeneracy assumption of Theorem~\ref{maincurve}, 
    \begin{align*}
        \{t\in\Lambda(\Qbar) \mid \exists z\in \Cscr_t(\Qbar), z\ \text{is periodic for}\ f_t \}\ \text{and}\ \{z\in \Cscr(\Qbar) \mid z\ \text{is periodic for}\ f \} 
    \end{align*}
    are sets of bounded height.
\end{thm}

\subsubsection{Dissipative families of quadratic H\'enon maps}
Let $\delta \in \K$ be such that $|\delta| <1.$ Consider the following family of quadratic H\'enon maps
\begin{align}\label{quadraticintro}
    f \colon (t,x,y)\in \K^3 \mapsto (t,y,y^2+t-\delta x)\in\K^3,
\end{align}
Denote by 
\begin{align}\label{ydeltat}
    y^\pm_{\delta,t}\coloneqq\frac{(1+\delta) \pm \sqrt{(1+\delta)^2 - 4t}}{2}.
\end{align}
Note that the two points $(y^+_{\delta,t},y^+_{\delta,t})$ and $(y^-_{\delta,t},y^-_{\delta,t})$ are the fixed points of $f_t$. Let $\Sigma_t$ be the set of the following 8 points
\begin{align}\label{Sigma_t}
  \Sigma_t\coloneqq \left\{ \begin{array}{l}
    (y^+_{\delta,t},y^+_{\delta,t}),(y^+_{\delta,t}-\delta,y^+_{\delta,t}),(y^+_{\delta,t},y^+_{\delta,t}-1)
    ,(y^+_{\delta,t}-\delta,y^+_{\delta,t}-1) \\
    (y^-_{\delta,t},y^-_{\delta,t}),(y^-_{\delta,t}-\delta,y^-_{\delta,t}),(y^-_{\delta,t},y^-_{\delta,t}-1)
    ,(y^-_{\delta,t}-\delta,y^-_{\delta,t}-1)
  \end{array}\right\}
\end{align}
Remark that all these points lie on the lines
\begin{align}\label{lines}
    y=x,\ \ y=x+1,\ \ y=x+\delta\ \ \mathrm{and} \ \ y=x+\delta-1,
\end{align}

\begin{thm}\label{maincurve2}
    
    Let $f_t$ be the family~\eqref{quadraticintro}. Let $\Cscr\subset \A\times\A^2$ be a family of curves parameterized by the affine line $\A^1$. Suppose there exists some positive constant $r>0$ such that for all $|t|$ large enough, the distance between $\Cscr_t$ and the eight points $\Sigma_t$ is larger than $r.$ Then there exist a positive constant $\varepsilon >0$ and a positive integer $N>0$ such that for all but finitely many $t\in \Lambda(\Qbar)$, the set $\{z\in \Cscr_t(\Qbar)\ |\ \hat{h}_{f_t}(z) \leq \varepsilon \}$ has at most $N$ points.

    In particular, for all but finitely many $t\in \Lambda(\C)$, there are at most $N$ periodic points on $\Cscr_t(\C).$
\end{thm}

The condition on the distance in Theorem~\ref{maincurve2} is rather mild. For example, if
$\Cscr$ is the constant family of a line $L\subset \K^2$, i.e. for any $t,$ $\Cscr_t=L$, which is not one of the lines in~\eqref{lines}, then it verifies the assumption.

\subsection{Equidistribution Theorems}\label{introequi}
\subsubsection{Other related works and Yuan-Zhang's equidistribution theorem}
In the one dimensional case, consider a pair of two families of rational maps $(f_t,g_t)$ parameterized by a quasi-projective curve defined over a number field. Then Mavraki and Schmidt \cite{MavrakiSchmidt} showed that there exists a positive integer $B$ such that, for all but finitely many parameters $t$, the number of common preperiodic points of $f_t$ is less than $B.$ This result was further generalized by DeMarco and Mavraki \cite{DemarcoMavraki24} to families parameterized by bases of higher dimensions (see also \cite{DKY1,DKY2}). Now consider the diagonal $\Delta\subset \P^1\times \P^1$ as the constant family of curves. In this case, all the common preperiodic points are precisely the preperiodic points of the map $(f_t,g_t)$ within the family of curves $\Delta.$

Let $\mathrm{Poly}^2_d$ denote the space of regular polynomial endomorphisms of degree $d\geq 2$ on the complex affine
plane $\C^2$ (i.e., which can be extended to an endomorphism of the projective plane $\P^2$ of degree $d$). For any $f\in \mathrm{Poly}^2_d$, let $C_f$ be the closure in $\P^2$ of critical locus of $f$ in $\C^2$. Then Gauthier, Taflin and Vigny \cite[Theorem D]{gauthier2023sparsity} established that there exist constants $B \geq1$,
$\varepsilon > 0$ and a non-empty Zariski open subset $U \subset \mathrm{Poly}^2_d$ such that for any $f \in U(\Qbar)$, the number of points of canonical height less than $\varepsilon$ in $C_f$ is bounded by $B.$ A crucial step in their work is \cite[Lemma 7.4]{gauthier2023sparsity}, where they assumed the existence of an open subset within the support of the \emph{bifurcation measure} (i.e., Green measure of the family of curves formed by the critical points). They subsequently the existebce of such an open subset in \cite[Lemma 7.6]{gauthier2023sparsity}, which is a highly non-trivial result and follows from one of their main theorems. However, for families of plane regular polynomial automorphisms, we are able to bypass this assumption in Proposition \ref{positivityI_f}.

These works can be viewed as analogues of various results in arithmetic geometry regarding uniform numbers of torsion points in families of abelian varieties. For example, the uniform Mordell-Lang conjecture for curves embedded into their Jacobians established by Dimitrov, Gao, Habegger and K\"uhne (\cite{DGH21,kuhne2021equidistribution}), and generalized to higher-dimensional subvarieties of abelian varieties by Gao, Ge and K\"uhne \cite{GGK21}. 

A key component in the aforementioned works is the equidistribution theorem for non-degenerate subvarieties. This approach traces back to the works of Szpiro, Ullmo and Zhang on Bogomolov's conjecture~\cite{SUZ97,Ullmo98,ZhangEquidistribution98}. More recently, Yuan and Zhang \cite{YZadelic} developed a theory of adelic line bundles on quasi-projective varieties. As an application, they proved an equidistribution theorem on quasi-projective varieties, extending K\"uhne's equidistribution theorem on families of abelian varieties \cite{kuhne2021equidistribution}. Additionally, Yuan \cite{yuan2023big} utilized their theory of adelic line bundles to provide an alternative proof of the uniform Mordell-Lang conjecture for curves and extended it to function fields of any characteristic. 

\subsubsection{Height inequalities ``à la Call-Silverman'' and equidistribution theorems for families of regular plane polynomial automorphisms}
We also rely on the general equidistribution theorem of Yuan and Zhang \cite{YZadelic}, in a reformulated form given by Gauthier \cite{gauthier2023good}.
To utilize this theorem, We first construct some geometric canonical height functions and interpret them as the mass of certains Green measures. We then establish some height inequalities. The first (Lemma \ref{heightinequalitylemma1}) --- which is analogous to Call-Silverman type height inequalities \cite{CallSilverman} --- is for families of plane regular polynomial automorphisms while the second (Lemma \ref{height inequality curve}) holds specifically on non-degenerate subvarieties. We verify that the conditions of the general equidistribution theorem are met in our context, allowing us to prove two specific equidistribution theorems~\ref{equidistributionpoints} (for marked points) and ~\ref{equidistribution curves} (for families of curves) for families of plane regular polynomial automorphisms.
\subsection{Outline}
In Section \ref{sectionrenormalization}, we prove our renormalization lemmas in a local setting. Section \ref{sectionmarkedpoints} focuses on the study of periodic points for marked points with proportional forward and backward Green measures, utilizing the renormalization results from Section \ref{sectionrenormalization}. In Section \ref{sectioncurves}, we construct geometric canonical heights for families of curves, demonstrate an important positivity property for non-degenerate curves, and estimate the degeneration of filled Julia sets for dissipative families of quadratic H\'enon maps. Section \ref{SectHeightInequalities} is dedicated to establishing the height inequalities mentioned earlier. In Section \ref{sectionequidistribution}, after covering preliminaries on adelic line bundles, we prove Theorem~\ref{MainBoundedHeight} and our equidistribution theorems for both marked points and families of curves. Finally, Section~\ref{sectionproof} integrates all the previously obtained results and provides proofs of our main theorems.

\subsection*{Acknowledgement}
I would like to thank my Ph.D. advisor Thomas Gauthier for introducing me to this field and for his constant support. I am also grateful to Romain Dujardin and Thomas Gauthier for advising me on this work. I would like to thank Charles Favre and Gabriel Vigny for answering my questions. I would like to thank Junyi Xie for a useful remark concerning Theorem~\ref{MainMarkedPoint2}.

\section{Renormalization lemmas}\label{sectionrenormalization}
In this section, we prove two renormalization lemmas: one for semi-repelling periodic points (Lemma \ref{renormalization semi}) and another for saddle periodic points (Lemma \ref{renormalization saddle}). The proofs are technical and will subsequently be used to establish two renormalization results for the fibered forward/backward Green functions $G_t^\pm$ in Section~\ref{sectionmarkedpoints} (Lemmas \ref{saddle parameters} and \ref{semiparameterslemma}).

Let us first introduce some notation.
Let $\Dbb(r)$ denote the open disk of radius $r$ centered at the origin of the complex plane. Fix a small positive real number $\varepsilon \ll 1$. Denote by $\pi$ the projection $\pi\colon \Dbb(1)^2 \times \Dbb(\varepsilon) \to \Dbb(\varepsilon)$. Let $a$ and $b$ be two holomorphic functions on $\Dbb(\varepsilon)$ defined by
\begin{align*}
    a(t) \coloneqq t^q + \mathrm{h.o.t.} \ \ \mathrm{and}\ \ b(t) \coloneqq t^p + \mathrm{h.o.t.}
\end{align*}
with $p,q \geq 1$. Let $\sigma(t) = (a(t), b(t))$. If $\gamma(t)=\sum_{i=0}^{+\infty}r_i t^i$ is a power series, then we define 
\begin{align*}
    \mdeg(\gamma) \coloneqq \min \left\{ i\in\N \mid r_i\neq 0  \right\}.
\end{align*}

\subsection{Semi-repelling renormalization}
Let $u,s \colon \Dbb(\varepsilon) \to \C$ be holomorphic functions such that we have $|s(0)|=1 , \sup|s(t)|<\inf|u(t)|$ and $\inf|u(t)|>1.$
Define a holomorphic family of holomorphic maps $f_t\colon\Dbb(1)^2 \to \Dbb(1)^2$ parameterized by $t\in \Dbb(\varepsilon)$ by 
\begin{align}\label{f_tsemi}
    f_t(x,y) \coloneqq (u(t)x + y\Tilde{u}_t(x,y), s(t)y + y\Tilde{s}_t(x,y)),
\end{align}
where $\Tilde{u}_t(x,y)$ and $\Tilde{s}_t(x,y)$ are power series in variables $x,y$ with coefficients holomorphic functions on $t$, such that $\Tilde{u}_t(0,0)=\Tilde{s}_t(0,0)=0.$ 

\begin{lem}[Semi-repelling renormalization]\label{renormalization semi}
    There exists a positive integer $D^+$ such that the following is true. Let $\lambda_u$ be a $D^+$-th root of $u(0)$. For all positive integers $n\geq 1$, we define the rescaling 
    factor $r_n \colon \Dbb(\varepsilon) \to \C$ by $r_n(t)\coloneqq t/ \lambda_u^n$.
For $0 \leq m\leq n$, let 
\begin{align*}
    f^m_{r_n(t)}(\sigma(r_n(t))) = \left(a_m(t) , b_m(t)\right).
\end{align*}
Then up to replacing $f_t$ by an iterate and up to shrinking $\varepsilon$, we have
\begin{align*}
    & \lim_{n\to \infty} a_n(t) = h(t),\ \ \lim_{n\to \infty} b_n(t) = 0\ \mathrm{and}\ \lim_{n\to +\infty}f^{-n}_{r_n(t)}(\sigma(r_n(t)))=(0,0)
\end{align*}
where $h$ is a non constant holomorphic function on the disk $\Dbb(\varepsilon)$ and the convergence is uniform.
\end{lem}

We now describe what the integer $D^+$ is.
If $x$ does not divide $\Tilde{u}_t(x,y)$, we can write 
\begin{align*}
    y\Tilde{u}_t(x,y)= \Tilde{u}^1_t(y) + \Tilde{u}^2_t(x,y)
\end{align*} 
with
\begin{align}\label{Tildeu^1_t}
    \Tilde{u}^1_t(y) = \sum_{k=2}^\infty c_k(t)y^k,
\end{align}
and $xy$ divides $\Tilde{u}^2_t(x,y)$.
Define $d_u\coloneqq \min_{k\geq 2} \left( \mdeg(c_k(t)t^{pk})\right)$ and
\[
    D^+\coloneqq  
\begin{cases}
    d_u,& \text{if } q> d_u\\
    q,              & \text{otherwise.}
\end{cases}
\]

\subsection{Sadlle renormalization}
In this subsection let $u,s\colon \Dbb(\varepsilon) \to \C$ be holomorphic functions such that $\sup|s(t)|<1<\inf|u(t)|$. Define a family of holomorphic maps $f_t\colon\Dbb(1)^2 \to \Dbb(1)^2$ parameterized by $t\in \Dbb(\varepsilon)$ by 
\begin{align}\label{f_tsaddle}
    f_t(x,y) = (u(t)x + xy\Tilde{u}_t(x,y), s(t)y + xy\Tilde{s}_t(x,y)),
\end{align}
where $\Tilde{u}_t(x,y)$ is a power series in variables $x,y$ with coefficients holomorphic functions on $t$, the same for $\Tilde{s}_t(x,y)$.

\begin{lem}[Saddle renormalization]\label{renormalization saddle}
    Denote by $\lambda_u$ a $q$-th root of $u(0)$ and $\lambda_s$ a $p-$th root of $s(0)$.
    For all integers $n\geq 0$, define $r_n \colon \Dbb(\varepsilon) \to \C$ by $r_n(t) = \frac{t}{\lambda_u^n}.$ For $0 \leq m\leq n$, let 
\begin{align*}
    f^m_{r_n(t)}(\sigma(r_n(t))) = \left(a_m(t) , b_m(t)\right).
\end{align*}
Then up to shrinking $\varepsilon$, we have $\lim_{n\to +\infty} a_n(t) = t^q$ and $\lim_{n\to +\infty} b_n(t) = 0$, and the convergence is uniform.

Moreover, if $|\lambda_u|>|\lambda_s^{-1}|$, then $\lim_{n\to +\infty} f^{-n}_{r_n(t)}(\sigma(r_n(t))) \to (0,0),$
and the convergence is uniform.
\end{lem}

The proof of Lemma \ref{renormalization saddle} is much simpler than that of Lemma \ref{renormalization semi}. This is essentially because of the fact that $|s(0)|<1$, indicating the presence of a stable manifold in the dynamical setting, which simplifies the local form of $f_t$. In essence, the saddle case can be viewed as a special case of the semi-repelling scenario. Therefore, we will focus on proving Lemma \ref{renormalization semi} and omit the proof of Lemma \ref{renormalization saddle}.

\subsection{Proof of Lemma \ref{renormalization semi}}

We prove the case where $xy$ is not a common factor of $y\Tilde{u}_t(x,y)$ and $q>d_u$. The other cases follow similarly. Note that in this proof, as we shrink $\varepsilon$,  the shrinking is done in a manner independent of $m$ and $n$, allowing us to consider only very large $n.$

Define $\hat{u}(t)$ and $\hat{s}(t)$ to be 
\begin{align}\label{hatuandhats}
    u(t)=u(0)+\hat{u}(t) \ \ \ \mathrm{and}\ \ \  s(t)=s(0)+\hat{s}(t).
\end{align}
Shrinking $\varepsilon$, we can assume that
\begin{align}\label{suphatu}
    \sup_{t\in \Dbb(\varepsilon)}|\hat{u}(t)|\leq 1 \ \ \ \mathrm{and}\ \ \ \sup_{t\in \Dbb(\varepsilon)}|\hat{s}(t)|\leq 1.
\end{align}
Let $ k_1 \leq \cdots\leq  k_l$ be all the integers such that $d_u = \mdeg\left(c_{k_i}(t)t^{pk_i}\right)$. Let $w$ be a complex variable, so that the coefficient 
\begin{align}\label{c_u}
    c_u(w)
\end{align}
of $t^{d_u}$ in $\sum_{i=1}^l c_{k_i}(t)(wt)^{pk_i}$ is a polynomial of $w$ of degree $pk_l$. Note that $c_u(w)$ is not identically zero by our assumption that $x$ does not divide $\Tilde{u}_t(x,y)$. For any $n \gg 1$, set 
$$n(p,d_u)\coloneqq\lfloor n(d_u-p)/d_u +1\rfloor +1. $$ 
It will be used for the equations~\eqref{small bs} and~\eqref{n(p,d)2}. Denote by $\lambda_s$ a $p$-th root of $s(0)$.

\subsubsection{Case A: Formulas of $b_m$ and $a_m$ for small $m$}\noindent

If $z\in \C$ is a complex number, then the notation $O_1(z)$ means that $|O_1(z)/z| \leq 1$.  We will show by induction that, if $0\leq m \leq n(p,d_u)$, then
\begin{align}\label{b_m Model}
    \boxed{b_m(t) = s(0)^m r_n^p(t) + \frac{3(m+1)}{\lambda_u^{n}} O_1(r_n^p(t))} \ .
\end{align}
Setting $a_0(t) \coloneqq a(r_n(t))$. If $1\leq m \leq n(p,d_u)+1$, then we will show that
\begin{align}\label{a_m Model}
    \boxed{a_m(t) = \sum_{i=0}^{m-1}\frac{c_u(\lambda_s^i)}{\lambda_u^{id_u}}\frac{t^{d_u}}{\lambda_u^{nd_u-(m-1)d_u}} +  O_1\left(\frac{\alpha_m}{\lambda_u^{n(d_u+1)-(m-1)d_u}}\right)}\ ,
\end{align}
where  $\alpha_m$ is defined by induction as follows.
Define $\alpha_1\coloneqq2+2\lambda_u^{-n}.$ Now if $\alpha_m$ is given and the formula \eqref{b_m Model} is true for $b_m(t)$, then we define $\alpha_{m+1}$ as follows. 
By the definition \eqref{Tildeu^1_t} of $\Tilde{u}^1_t$, we have
\begin{align*}
     \Tilde{u}^1_{r_n(t)}(b_m(t)) &=\sum_{i=1}^l c_{k_i}(r_n(t)) \Big(b_m(t)\Big)^{k_i} + \sum_{k\neq k_i}c_{k}(r_n(t))\Big(b_m(t)\Big)^{k}.
\end{align*}
By the definition of $\lambda_s$ and \eqref{b_m Model}, we have
\begin{multline}
    \Tilde{u}^1_{r_n(t)}(b_m(t))=\sum_{i=1}^l c_{k_i}(r_n(t)) \Big(\lambda_s^m r_n(t)\Big)^{pk_i} \left(1+  \frac{3(m+1)}{s(0)^{m}\lambda_u^{n}} O_1(1)\right)^{k_i} \\
    + \sum_{k\neq k_i}c_{k}(r_n(t))\left( r^p_n(t)   O_1\left(s(0)^m+ \frac{3(m+1)}{\lambda_u^{n}}\right)\right)^k \eqqcolon I_1(m,n) + I_2(m,n).
\end{multline}
By the definition \eqref{c_u} of $c_u$ and the fact there are only finitely many $k_i$, there exists a positive constant $c_1>0$, independent of $m$ and $n$, such that 
\begin{align*}
    I_1(m,n)=c_u(\lambda_s^m)r_n^{d_u}(t) + \frac{c_1 3(m+1)}{\lambda_u^{n}}O_1(r^{d_u}_n(t)).
\end{align*}
Since $t^{d_u+1}\lambda_u^{-n(d_u+1)}$ divides every term of $I_2(m,n)$, shrinking $\varepsilon$, there exists a positive constant $c_2>0$, independent of $m$ and $n$, such that 
\begin{align*}
   I_2(m,n) =\frac{c_2}{\lambda_u^n}O_1(r^{d_u}_n(t)).
\end{align*}
Setting $c_3 \coloneqq c_1+c_2/(3(m+1))$, We have
\begin{align}\label{utilde1}
\begin{split}
    \Tilde{u}^1_{r_n(t)}(b_m(t)) &= c_u(\lambda_s^m)r_n^{d_u}(t) + \frac{c_3 3(m+1)}{\lambda_u^{n}}O_1(r^{d_u}_n(t))\\
    &=\frac{c_u(\lambda_s^m)}{\lambda_u^{md_u}}\frac{t^{d_u}}{\lambda_u^{nd_u-md_u}}
    + \frac{c_3 3(m+1)\varepsilon^{d_u}}{\lambda_u^{md_u}}O_1\left(\frac{1}{\lambda_u^{n(d_u+1)-md_u}}\right).    
\end{split}
\end{align}
Since $\lim_{m\to \infty}(m+1)\lambda_u^{-md_u} = 0$, we can shrink $\varepsilon$ so that, for all $m$ and $n$ large enough, we have
\begin{align}\label{epsilon(n,m)}
    \varepsilon(n,m)\coloneqq\varepsilon^{d_u} + \frac{1}{\lambda_u^{n}} +\frac{c_3 3(m+1)}{\lambda_u^{md_u}} \varepsilon^{d_u} + \varepsilon < \frac{|\lambda_u| -1}{2}.
\end{align}
We define 
\begin{align}\label{alpha_m+1}
    \alpha_{m+1} \coloneqq \alpha_m(1+\varepsilon(n,m)).
\end{align}
Note that we have the following estimate
\begin{align}\label{boundalpha1}
    \left|\prod_{i=1}^{n(p,d_u)}\alpha_i\right| = \left|\alpha_1\right|{\left(\frac{|\lambda_u|+1}{2}\right)}^{n(p,d_u)}\leq\left|\alpha_1\right|{|\lambda_u|}^{n(p,d_u)} =\frac{|\alpha_1|}{|\lambda_u|^{n-n(p,d_u)}}{|\lambda_u|}^{n}.
\end{align}

\medskip

\subsubsection{Proof of Case A}\noindent

By the definition of $b(t)$, there is a constant $c_{b_0}$ such that
\begin{align*}
    b_0(t) = b(r^p_n(t)) = r^p_n(t) + c_{b_0}\frac{t}{\lambda_u^{n}}(r^p_n(t) +\mathrm{h.o.t.} ),
\end{align*}
shrinking $\varepsilon,$ we can make sure that $|c_{b_0}\varepsilon(r^p_n(t) +\mathrm{h.o.t.})/(3r^p_n(t))| \leq 1,$ so that $b_0(t)$ has the wanted form \eqref{b_m Model}
\begin{align*}
    \boxed{b_0(t) =  r^p_n(t) + \frac{3}{\lambda_u^n}O_1(r^p_n(t))}\ .
\end{align*}

\medskip

Now we compute $a_1(t)$. By the definition \eqref{f_tsemi} of $f_t,$ we have
\begin{align*}
    a_1(t) &= u(r_n(t))a(r_n(t)) + \Tilde{u}^1_{r_n(t)}\left( b(r_n(t))\right) +\Tilde{u}^2_{r_n(t)}\left( a(r_n(t)) , b(r_n(t))\right).
\end{align*}
The first term  is
\begin{align*}
u(r_n(t))a(r_n(t))=\left(\lambda_u^{d_u} + O_1\left(\frac{1}{\lambda_u^{n}}\right)\right) \left(\frac{t^q}{\lambda_u^{nq}}(1+\mathrm{h.o.t})\right).
\end{align*}
We can reduce $\varepsilon$ so that $| \lambda_u^{d_u} \varepsilon^q(1+\mathrm{h.o.t})| \leq 1$. Since moreover $q>d_u$, we obtain that
\begin{align*}
    u(r_n(t))a(r_n(t)) = O_1\left(\frac{1}{\lambda_u^{n(d_u+1)}}\right) + \frac{1}{\lambda_u^{n}} O_1\left(\frac{1}{\lambda_u^{n(d_u+1)}}\right).
\end{align*}
By \eqref{utilde1}, shrinking $\varepsilon$, the second term is 
\begin{align*}
    \Tilde{u}^1_{r_n(t)}\left( a(r_n(t)) , b(r_n(t))\right)= c_u(1)\frac{t^{d_u}}{\lambda_u^{nd_u}} + O_1\left(\frac{1}{\lambda_u^{n(d_u+1)}}\right).
\end{align*}
Since $xy$ divides $\Tilde{u}^2_{t}(x,y)$, shrinking $\varepsilon,$ the third term is 
\begin{align*}
    \Tilde{u}^2_{r_n(t)}( a(r_n(t)) , b(r_n(t)))=O_1\left(\frac{1}{\lambda_u^{n(q+p)}}\right) = \frac{1}{\lambda_u^{n}}O_1\left(\frac{1}{\lambda_u^{n(d_u+1)}}\right).
\end{align*}
Adding the three terms, we obtain finally that
\begin{align*}
    \boxed{a_1(t) = c_u(1)\frac{t^{d_u}}{\lambda_u^{nd_u}}  + (2+\frac{2}{\lambda_u^{n}})O_1\left(\frac{1}{\lambda_u^{n(d_u+1)}}\right)}\ .
\end{align*}\\

Suppose $a_m(t)$ and $b_m(t)$ have the form \eqref{a_m Model} and \eqref{b_m Model} respectively. Let us first compute $b_{m+1}(t)$:
\begin{align*}
    b_{m+1}(t) = s(r_n(t))b_m(r_n(t)) + b_m(t)\Tilde{s}_{r_n(t)}(a_m(r_n(t)), b_m(r_n(t)) )
\end{align*}
Since $|r^p_n(\varepsilon)| + |3n\lambda_u^{-n}O_1(r^p_n(\varepsilon))|$ is uniformly bounded on $n$, and $t$ divides $\hat{s}(t)$ (recall \eqref{hatuandhats}), shrinking $\varepsilon$, we have
\begin{align*}
    \hat{s}(r_n(t))\left(r^p_n(t) + \frac{3n}{\lambda_u^{n}}O_1(r^p_n(t))\right) = \frac{1}{\lambda_u^{n}}O_1(r^p_n(t)).
\end{align*}
Hence the first term $s(r_n(t))b_m(t)$ is
\begin{align}\label{caseAb_m1}
\begin{split}
    &\ \ \ \Big(s(0) + \hat{s}(r_n(t))\Big) \left( s(0)^m r_n^p(t) + \frac{3(m+1)}{\lambda_u^{n}} O_1(r_n^p(t)\right)\\
    &=s(0)^{m+1} r_n^p(t) + \frac{3(m+1)}{\lambda_u^{n}} O_1(r_n^p(t)) + \frac{1}{\lambda_u^{n}} O_1(r_n^p(t)).
\end{split}
\end{align}
Since $np\leq nd_u-(m-1)d_u,$ if $m\geq 1,$ by the definition \eqref{a_m Model} of $a_m(t)$ and the estimate~\eqref{boundalpha1},
\begin{align}\label{small bs}
    |a_m(t)| \leq \sum_{i=0}^{m-1}\left|\frac{c_u(\lambda_s^i)}{\lambda_u^{id_u}}\right|\frac{t^{d_u}}{\lambda_u^{np}} +  \frac{|\alpha_1|}{|\lambda_u|^{n-n(p,d_u)}}O_1\left(\frac{1}{\lambda^{np}_u}\right) = O_1\left(\frac{\mathscr{E}_{\varepsilon,n}}{\lambda^{np}_u}\right)
\end{align}
where $\mathscr{E}_{\varepsilon,n} \to 0$ when $ \varepsilon \to 0$ and $n\to +\infty$.
If $m=0,$ by our assumption $q>d_u\geq p$, we have
\begin{align}\label{equation4}
a_0(t)=a(r_n(t))= O_1\left(2r_n^p(t)\right).
\end{align}
It is also clear by \eqref{b_m Model} that 
\begin{align}\label{equation3}
    b_m(t)=O_1\left(2r_n^p(t)\right).
\end{align}
Since $\Tilde{s}_t(0,0)=0$, by \eqref{small bs}, \eqref{equation4} and \eqref{equation3}, we have (for $n\gg 1$ and $\varepsilon\ll 1$)
\begin{align*}
     \Tilde{s}_{r_n(t)}(a_m(t),b_m(t)) = O_1\left(\frac{1}{\lambda^{np}_u}\right).
\end{align*}
Hence
\begin{align}\label{caseAb_m2}
    b_m(t)\Tilde{s}_{r_n(t)}(a_m(t),b_m(t)) =\frac{1}{\lambda_u^{np}} O_1(r^p_n(t))=\frac{1}{\lambda_u^{n}}O_1(r^p_n(t)).
\end{align}
Adding \eqref{caseAb_m1} and \eqref{caseAb_m2}, we obtain
\begin{align*}
    \boxed{ b_{m+1}(t) = s(0)^{m+1} r_n^p(t) + \frac{3(m+2)}{\lambda_u^{n}} O_1(r_n^p(t)).}
\end{align*}\\

Let us compute now $a_{m+1}(t)$:
\begin{align*}
    a_{m+1}(t)=u(r_n(t))a_m(t) +\Tilde{u}^1_{r_n(t)}\left( b_m(t)\right) +\Tilde{u}^2_{r_n(t)}\left( a_m(t)) , b_m(t)\right)
\end{align*}
By \eqref{suphatu}, $\hat{u}(r_n(t)) = O_1\left(\frac{1}{\lambda_u^n}\right)$. Thus the first term $u(r_n(t))a_m(t)$ is equal to
\begin{align}\label{caseAa_m1}
\begin{split}
    &\ \ \ (u(0) + \hat{u}(r_n(t))) \left(\sum_{i=0}^{m-1}\frac{c_u(\lambda_s^i)}{\lambda_u^{id_u}}\frac{t^{d_u}}{\lambda_u^{nd_u-(m-1)d_u}} +  O_1\left(\frac{\alpha_m}{\lambda_u^{n(d_u+1)-(m-1)d_u}}\right) \right)\\
    &=\sum_{i=0}^{m-1} \frac{c_u(\lambda_s^i)}{\lambda_u^{id_u}}\frac{t^{d_u}}{\lambda_u^{nd_u-md_u}} + O_1\left(\frac{\alpha_m}{\lambda_u^{n(d_u+1)-md_u}}\right)\\
    &+ \varepsilon^{d_u} O_1\left(\frac{1}{\lambda_u^{n(d_u+1)-(m-1)d_u}}\right) + \frac{\alpha_m}{\lambda^{n}}O_1\left(\frac{1}{\lambda_u^{n(d_u+1)-md_u}}\right).
\end{split}
\end{align}
Since $xy$ divides $\Tilde{u}^2_{t}(x,y)$, we can reduce $\varepsilon$ so that
\begin{align}\label{caseAa_m2}
    \Tilde{u}^2_{r_n(t)}(a_m(t),b_m(t)) = \varepsilon O_1\left(\frac{1}{\lambda_u^{n(d_u+1)-md_u}}\right).
\end{align}
Summing up \eqref{caseAa_m1}, \eqref{caseAa_m2} and \eqref{utilde1}, we have
\begin{align*}
    a_{m+1}(t) &= \sum_{i=0}^m \frac{c_u(\lambda_s^i)}{\lambda_u^{id_u}} \frac{t^{d_u}}{\lambda_u^{nd_u-md_u}}\\
    &\ \ \ +\left(\alpha_m+\epsilon^{d_u}+\frac{\alpha_m}{\lambda_u^{n}}+\frac{c_3 3(m+1)\varepsilon^{d_u}}{\lambda_u^{md_u}}   +\varepsilon\right)O_1\left(\frac{1}{\lambda_u^{n(d_u+1)-md_u}}\right)
\end{align*}
By the definition \eqref{alpha_m+1} of $\alpha_{m+1}$, we obtain finally that
\begin{align*}
    \boxed{a_{m+1} = \sum_{i=0}^m \frac{c_u(\lambda_s^i)}{\lambda_u^{id_u}} \frac{t^{d_u}}{\lambda_u^{nd_u-md_u}} + \alpha_{m+1}O_1\left(\frac{1}{\lambda_u^{n(d_u+1)-md_u}}\right)}\ .
\end{align*}\\

\subsubsection{Case B: Formulas of $b_m$ and $a_m$ for large $m$}\noindent

The next step is to give formulas of $a_{m+1}(t)$ and $b_m(t)$ for $m > n(p,d_u).$

For $m\in \N$, set $m'\coloneqq n(p,d_u)+m$. Define a sequence $\beta_{m'}$ as follows. For $m=0,$ define $\beta_{0'} = 3(n(p,d_u)+1)\lambda_u^{-n}$; then for $m\geq 1$, define 
\begin{align}\label{beta_m'}
\begin{split}
    \beta_{m'} &= \beta_{(m-1)'} + \frac{3}{\lambda_u^{n}}+\frac{\lambda_u^{(m'-2)d_u}}{\lambda_u^{nd_u}} =\frac{3(m+1)'}{\lambda_u^{n}} + \frac{\sum_{i=1}^m \lambda_u^{(i'-2)d_u}}{\lambda_u^{nd_u}}.
\end{split}
\end{align}
Now we define $\alpha_{m'}$ for $m\geq 1.$
The same reason as for \eqref{utilde1} implies that, up to shrinking $\varepsilon$, there exists a positive constant $c_4>0$, independent on $m$ and $n$, such that
\begin{align}\label{u1'}
\begin{split}
    \Tilde{u}^1_{r_n(t)}(b_{m'}(t)) &= c_u(\lambda_s^{m'})\frac{t^{d_u}}{\lambda_u^{nd_u}} + c_4 \beta_{m'} O_1(r^{d_u}_n(t))\\
    &=\frac{c_u(\lambda_s^{m'})}{\lambda_u^{m'd_u}}\frac{t^{d_u}}{\lambda_u^{nd_u-m'd_u}}
    + c_4 \beta_{m'}\lambda_u^{n-m'd_u}\varepsilon^{d_u}O_1\left(\frac{1}{\lambda_u^{n(d_u+1)-m'd_u}}\right).
\end{split}
\end{align}
Since 
\begin{align*}
    \beta_{m'}\lambda_u^{n-m'd_u} =  \frac{3(m+1)'}{\lambda_u^{m'd_u}} + \frac{1}{\lambda_u^{n(d_u-1)}\sum_{i=1}^m \lambda_u^{(m-i+2)d_u}}
\end{align*}
is uniformly bounded on $m$ and $n$, we can shrink $\varepsilon$ so that
\begin{align}\label{epsilon'}
    \varepsilon(n,m')\coloneqq2\varepsilon + \lambda_u^{-n} + c_4\beta_{m'}\lambda_u^{n-m'd_u}\varepsilon^{d_u} \leq \frac{|\lambda_u| -1}{2}.
\end{align}
For $m\geq 1,$ define 
\begin{align}\label{alpha_m'}
    \alpha_{m'+1} = \alpha_{m'}(1+\varepsilon(n,m')).
\end{align}
Then by \eqref{alpha_m+1} and \eqref{alpha_m'}, we have
\begin{align*}
    |\alpha_{m'}|\leq \alpha_{(n-n(p,d_u))'} = \alpha_n = \alpha_0\prod_{m=1}^n(1+\varepsilon(n,m)).
\end{align*}
By \eqref{epsilon(n,m)} and \eqref{epsilon'}, $\alpha_n \leq \alpha_0 {\left(\left(|\lambda_u| +1\right)/2\right)}^{n}.$ Thus
\begin{align}\label{Useepsilon(n,m)}
    \bigg| \frac{\alpha_n}{\lambda_u^{n}}\bigg|  \leq \alpha_0 \left(\frac{|\lambda_u| +1 }{2\lambda_u}\right)^{n} \longrightarrow 0, \ \ \ n\to +\infty
\end{align}

Let us show by induction that if $0 \leq m \leq n-n(p,d_u)$, then
\begin{align}\label{b' model}
    \boxed{b_{m'}(t) = s(0)^{m'}r^p_n(t) + \beta_{m'} O_1(r^p_n(t))}\ ;
\end{align}
and if $1 \leq m \leq n-n(p,d_u)$, then
\begin{align}\label{a' model}
    \boxed{a_{m'}(t) =  \sum_{i=0}^{m'-1}\frac{c_u(\lambda_s^i)}{\lambda_u^{id_u}}\frac{\lambda_u^{(m'-1)d_u}}{\lambda_u^{nd_u}} t^{d_u}+\alpha_{m'} O_1\left(\frac{\lambda_u^{(m'-1)d_u}}{\lambda_u^{n(d_u+1)}}\right)}\ .
\end{align}\\

\subsubsection{Proof of Case B}\noindent

For $m=0$, \eqref{b' model} is exactly \eqref{b_m Model}. Now if \eqref{b' model} is true for $b_{m'},$ for some $0 \leq m \leq n-n(p,d_u)-1$, we will show \eqref{b' model} for $b_{(m+1)'}$:
\begin{align*}
    b_{m'+1}(t) &= s(r_n(t))b_{m'}(t) + b_{m'}(t)\Tilde{s}_{r_n(t)}(a_{m'}(t), b_{m'}(t) ).
\end{align*}
By construction \eqref{beta_m'}, $ \beta_{m'}$ is bounded. Hence shrinking $\varepsilon$, we have
\begin{align*}
    \hat{s}(r_n(t))b_{m'}(t)=\frac{1}{\lambda_u^{n}}  O_1(r^p_n(t)).
\end{align*}
Hence
\begin{align}\label{caseBb_m1}
\begin{split}
    s(r_n(t))b_{m'}(t)&= s(0) \left(s(0)^{m'} r_n^p(t) + \beta_{m'} O_1(r_n^p(t)) \right) +  \Tilde{s}(r_n(t))b_{m'}(t)\\
    &=s(0)^{m'+1} r_n^p(t) + \beta_{m'} O_1(r_n^p(t)) +\frac{1}{\lambda^{n}}O_1(r_n^p(t)).   
\end{split}
\end{align}
Since 
\begin{align}\label{n(p,d)2}
    np-nd_u + (m'-1)d_u\geq md>1,
\end{align}
By~\eqref{Useepsilon(n,m)} and the fact that $\Tilde{s}_t(0,0)=0$, we have (for $n\gg 1$ and $\varepsilon \ll 1$),
\begin{align*}
    \Tilde{s}_{r_n(t)}(a_{m'}(t), b_{m'}(t))=O_1\left(\frac{\lambda_u^{(m'-1)d_u}}{\lambda_u^{nd_u}}\right).
\end{align*}
Hence
\begin{align}\label{caseBb_m2}
    b_{m'}(t)\Tilde{s}_{r_n(t)}(a_{m'}(t), b_{m'}(t) )=\frac{\lambda_u^{(m'-1)d_u}}{\lambda_u^{nd_u}}O_1(r^p_n(t)).
\end{align}
Summing up \eqref{caseBb_m1} and \eqref{caseBb_m2}, we obtain
\begin{align*}
    b_{m+1}(t) &= s(0)^{m'+1} r_n^p(t) +\left(\beta_{m'} +\frac{1}{\lambda^{n}}  +\frac{\lambda_u^{(m'-1)d_u}}{\lambda_u^{nd_u}}\right)O_1(r^p_n(t))\\
    &= s(0)^{m'+1} r_n^p(t) +(\beta_{m'} + \left(\beta_{m'} +\frac{3}{\lambda^{n}}  +\frac{\lambda_u^{(m'-1)d_u}}{\lambda_u^{nd_u}}\right)O_1(r^p_n(t)).
\end{align*}
By the definition \eqref{beta_m'} of $\beta_{m'+1}$, we get finally 
\begin{align*}
    \boxed{b_{m'+1}(t) = s(0)^{m'+1} r_n^p(t) + \beta_{m'+1}O_1(r^p_n(t))}\ .
\end{align*}\\

For $m\geq 1,$ suppose that $a_{m'}(t)$ and $b_{m'}(t)$ have the forms \eqref{a' model} and \eqref{b' model}, let us compute $a_{(m+1)'}(t)$:
\begin{align*}
    a_{(m+1)'}(t) = u(r_n(t))a_{m'}(t) + \Tilde{u}_{r_n(t)}^1(b_{m'}(t)) + \Tilde{u}_{r_n(t)}^2(a_{m'}(t) , b_{m'}(t)).
\end{align*}
The first term $u(r_n(t))a_{m'}(t)$ is equal to
\begin{align*}
    &\ \ \ \big(u(0)+\hat{u}(r_n(t))\big)\left(\sum_{i=0}^{m'-1}\frac{c_u(\lambda_s^i)}{\lambda_u^{id_u}}\frac{\lambda_u^{(m'-1)d_u}}{\lambda_u^{nd_u}} t^{d_u}
    +\alpha_{m'} O_1\left(\frac{\lambda_u^{(m'-1)d_u}}{\lambda_u^{n(d_u+1)}}\right)\right)\\
    &=\sum_{i=0}^{m'-1}\left(\frac{c_u(\lambda_s^i)}{\lambda_u^{id_u}}\frac{\lambda_u^{m'd_u}}{\lambda_u^{nd_u}}\right) t^{d_u} 
    + \alpha_{m'}O_1\left(\frac{\lambda_u^{m'd_u}}{\lambda_u^{n(d_u+1)}}\right)
    +\varepsilon O_1\left(\frac{\lambda_u^{(m'-1)d_u}}{\lambda_u^{n(d_u+1)}}\right)+ \frac{\alpha_{m'}}{\lambda^{n}} O_1\left(\frac{\lambda_u^{(m'-1)d_u}}{\lambda_u^{n(d_u+1)}}\right).
\end{align*}
By \eqref{a' model}, \eqref{b' model} and the fact that $\Tilde{u}^2_t(0,0)=0$, we can reduce $\varepsilon$ so that 
\begin{align*}
    \Tilde{u}^2_{r_n(t)}(a_{m'}(t), b_{m'}(t))=\varepsilon O_1\left(\frac{\lambda_u^{m'd_u}}{\lambda_u^{n(d_u+1)}}\right),
\end{align*}
Thus adding the above two terms and \eqref{u1'}, we obtain that
\begin{align*}
    a_{m'+1}(t) &= \sum_{i=0}^{m'}\frac{c_u(\lambda_s^i)}{\lambda_u^{-id_u}} \frac{\lambda_u^{m'd_u}}{\lambda_u^{nd_u}}t^{d_u}
    +\left(\alpha_{m'} +2\varepsilon +\frac{\alpha_{m'}}{\lambda_u^{n}}+c_4\beta_{m'}\lambda_u^{n-m'd_u}\varepsilon^{d_u}\right) O_1\left(\frac{\lambda_u^{m'd_u}}{\lambda_u^{n(d_u+1)}}\right).
\end{align*}
Since $\alpha_{m'}>1$, and by the definition \eqref{alpha_m'} of $\alpha_{(m+1)'}$, we have
\begin{align*}
    &\ \ \ \alpha_{m'} +2\varepsilon +\frac{\alpha_{m'}}{\lambda_u^{n}}+c_3\beta_{m'}\lambda_u^{n-m'd_u}\varepsilon^{d_u}\\
    &=\alpha_{m'}\left( 1+2\varepsilon + \frac{1}{\lambda_u^{n}} + c_3\beta_{m'}\lambda_u^{n-m'd_u}\varepsilon^{d_u}\right)\\
    &=\alpha_{m'}(1+\varepsilon(n,m'))\\
    &=\alpha_{m'+1},
\end{align*}
Thus we get finally
\begin{align*}
    \boxed{a_{m'+1}(t) = \sum_{i=0}^{m'}\frac{c_u(\lambda_s^i)}{\lambda_u^{-id_u}} \frac{\lambda_u^{m'd_u}}{\lambda_u^{nd_u}}t^{d_u}+ \alpha_{m'+1}O_1\left(\frac{\lambda_u^{m'd_u}}{\lambda_u^{n(d_u+1)}}\right)}\ .
\end{align*}\\

\subsubsection{Convergence of $a_m$ and $b_m$}
Let $m=n-n(p,d_u)$, by \eqref{a' model}, we obtain that
\begin{align*}
    a_n(t) = a_{m'}(t) = \sum_{i=0}^{n-1} \frac{c_u(\lambda_s^i)}{\lambda_u^{(i+1)d_u}}t^{d_u} + \frac{\alpha_n}{\lambda_u^{n}} O_1\left(\frac{1}{\lambda_u^{d_u}}\right).
\end{align*}

It implies  that
\begin{align}\label{limita_n}
    \lim_{n\to \infty} a_n(t) = \sum_{i=0}^{\infty}\frac{c_u(\lambda_s^i)}{\lambda_u^{(i+1)d_u}} t^{d_u}.
\end{align}
It's also straightforward to see that $\lim_{n\to \infty} b_n(t)=0$.
Finally we obtain that
\begin{align*}
    \boxed{\lim_{n\to \infty}f^n_{r_n(t)}(\sigma(r_n(t)))= \big(\lim_{n\to \infty} a_n(t),0\big)}\ .
\end{align*}\\

\subsubsection{$\lim_n a_n(t)$ is non constant}\noindent

Let us show that up to taking an iterate of $f_t$, the function \eqref{limita_n} is not constant.
By~\eqref{c_u}, we can write $c_u(w) = \sum_{j=2}^{d_u}e_jw^{pj}$, where $e_j$ is a complex number. 
By the geometric sum formula,
\begin{align*}
    \sum_{i=0}^{\infty}\frac{c_u(\lambda_s^i)}{\lambda_u^{
    (i+1)d_u}} &= \sum_{j=2}^{d_u} e_j \sum_{i=0}^{\infty}\frac{s(0)^{ji}}{\lambda_u^{(i+1)d}} =\sum_{j=2}^{d_u}\frac{e_j}{\lambda_u^{d_u} - s(0)^{j}}\\
    &=\frac{(\sum_{j=2}^{d_u}e_j)(\lambda_u^{d_u})^{d_u-2}+\sum_{i=3}^{d_u}L_i((e_j)_j)(\lambda_u^{d_u})^{d_u-i}}{\prod_{j=2}^{d_u}(\lambda_u^{d_u} - s(0)^j)} \numberthis \label{numerator},
\end{align*}
where $L_i((e_j)_j)$ is a linear form on $(e_j)_j$ with coefficients some powers of $s(0)$. It suffices to show that after taking a large iterate of $f_t$, the numerator of~\eqref{numerator} is non vanishing. 

Setting 
\begin{align*}
    \Tilde{U}^2_t(x,y)\coloneqq  \Tilde{u}^2_t(f_t(x,y))=\Tilde{u}^2_t(u(t)x + y\Tilde{u}_t(x,y),s(t)y+y\Tilde{s}_t(x,y)),
\end{align*}
and
\begin{align*}
    \Tilde{U}^1_t(x,y)&\coloneqq\Tilde{u}^1_t(s(t)y+y\Tilde{s}_t(x,y)) =\sum_{k=2}^\infty c_k(t)(s(t)y+y\Tilde{s}_t(x,y))^k\\
    &=\sum_{i=1}^l c_{k_i}(t)s(0)^{k_i}y^{k_i} + \Tilde{V}^1_t(y) + \Tilde{V}^2_t(x,y),
\end{align*}
where $\mdeg \left(\Tilde{V}^1_t(t^p)\right)> d_u$ and $x$ divides $\Tilde{V}^2_t(x,y).$
We have
\begin{align*}
    u(t)y\Tilde{u}_t(x,y)=u(t)(\Tilde{u}^1_t(y) + \Tilde{u}^2_t(x,y))=\sum_{k=2}^\infty u(0)c_k(t)y^k +\hat{u}(t)\Tilde{u}_t^1(y)+u(t)\Tilde{u}_t^2(x,y).
\end{align*}
Thus the first coordinate of $f_t^2(x,y)$ is equal to
\begin{multline*}
        u(t)\Big(u(t)x + y\Tilde{u}_t(x,y)\Big) + \Tilde{U}^1_t(y) + \Tilde{U}^2_t(x,y)\\
    =u(t)^2x+ \sum_{i=1}^l c_{k_i}(t)\big(u(0)+s(0)^{k_i}\big)y^{k_i} +V_t^{1,1}(y)+V_t^{2,1}(x,y).
\end{multline*}
Here $\mdeg \left(V_t^{1,1}(t^p) \right)>d_u$ and $x$ divides $V_t^{2,1}(x,y)$ By induction, for any $l\in \Z_{\geq 1}$, the first coordinate of $f^{2^l}_t(x,y)$ is equal to
\begin{align*}
    u(t)^{2^k}x + \sum_{i=1}^l c_{k_i}(t)\prod_{i=0}^{l-1}\big(u(0)^{2^i}+s(0)^{2^ik_i}\big)y^{k_i}+V_t^{1,l}(y)+V_t^{2,l}(x,y),
\end{align*}
where $\mdeg \left(V_t^{1,l}(t^p) \right)>d_u$ and $x$ divides $V_t^{2,l}(x,y)$.

We just showed that taking an iteration $f^{2^l}_t$ will change $e_j$ to $\prod_{i=0}^{l-1}\big(u(0)^{2^i}+s(0)^{2^ik_i}\big) e_j$, $\lambda_u$ to $\lambda_u^{2^l}$ and $s(0)$ to $s(0)^{2^l}$, but the degree $d_u$ stays invariant. Thus the numerator of the fraction~\eqref{numerator} becomes
\begin{align*}
    \prod_{i=0}^{l-1}\left(u(0)^{2^i}+s(0)^{2^ik_i}\right) \left( (\sum_{j=2}^{d_u}e_j)(\lambda_u^{2^ld_u})^{d_u-2}+\sum_{i=3}^{d_u}L_i((e_j)_j)(\lambda_u^{2^ld_u})^{d_u-i}  \right)
\end{align*}
Thus for $l$ large enough, it does not vanish.

\subsubsection{Vanishing of $\lim_{n\to +\infty}f^{-n}_{r_n(t)}(\sigma(r_n(t)))$}\noindent

Let us restate the question. 
Let $u,s \colon \Dbb(\varepsilon) \to \C$ be holomorphic functions such that, 
\begin{align*}
     |s(0)|=1,\ \   \sup|u(t)|<1,\ \ \mathrm{and}\ \  \sup|u(t)|<\inf|s(t)|.
\end{align*}
Define a family of holomorphic maps $f_t\colon\Dbb(1)^2 \to \Dbb(1)^2$ parameterized by $t\in \Dbb(\varepsilon)$ by 
\begin{align*}
    f_t(x,y) = (u(t)x + y\Tilde{u}_t(x,y), s(t)y + y\Tilde{s}_t(x,y)),
\end{align*}
where $\Tilde{u}_t(x,y)$ is a power series in variables $x,y$ coefficients holomorphic function on $t$, such that $\Tilde{u}_t(0,0)=0,$ the same for $\Tilde{s}_t(x,y)$.
Let $\lambda$ be a complex number such that $|\lambda|>1.$ Define $r_n(t)\coloneqq\lambda^n t$. Fix $n\gg 1$, as above, we define $a_m(t)$ and $b_m(t)$ by 
\begin{align*}
    f^m_{r_n(t)}(\sigma(r_n(t)))=(a_m(t),b_m(t)).
\end{align*}
For $1\leq m\leq n$, define 
\[
S_m\coloneqq\left(\frac{|\lambda|+1}{2} +\sum_{j=1}^{m-1} \left(\frac{|\lambda|+1}{2}\right)^j\right)\bigg/|\lambda|^n.
\] 
To show $\lim_{n\to +\infty}f^{n}_{r_n(t)}(\sigma(r_n(t)))$, it suffices to show that
\begin{align}\label{vanishing_model_a_m}    |a_m(t)|\leq S_m|t|,
\end{align}
and
\begin{align}\label{vanishing_model_b_m}
    |b_m(t)|<|t|\bigg| \frac{|\lambda|+1}{2}\bigg|^m\bigg/|\lambda|^n.
\end{align}
Shrinking $\varepsilon,$ since $|u(0)|<1$, we can let $|u(t)a(t)|+|b(t)\Tilde{u}_s(a(t),b(t))|<|t|,$
so that \eqref{vanishing_model_a_m} is true for $m=1.$
There exists $\varepsilon'\ll 1$ such that
\begin{align*}
    \sup_{t\in\Dbb(\varepsilon')}|s(t)|+\sup_{t,x,y\in\Dbb(\varepsilon')}|\Tilde{s}_t(x,y)| < \left|\frac{|\lambda|+1}{2}\right|.
\end{align*}
Let $\varepsilon\ll\varepsilon'$
Then for $t,x,y \in \Dbb(\varepsilon')$, 
\begin{align}\label{vanishing_inequality_b_m}
    |s(t)y+y\Tilde{s}_t(x,y)| < \bigg|\frac{|\lambda|+1}{2} y\bigg|,
\end{align}
so that \eqref{vanishing_model_b_m} is true for $m=1$

Suppose $\eqref{vanishing_model_a_m}$ and \eqref{vanishing_model_b_m} are true for $m.$ 
Shrink $\varepsilon'$ such that we have $\sup_{x,y\in \Dbb(\varepsilon')}|\Tilde{u}_t(x,y)| <1.$ Shrink $\varepsilon$ so that for all $1\leq m\leq n$, we have $S_m \varepsilon \leq \varepsilon'$ and 
\[|t|\bigg| \frac{|\lambda|+1}{2}\bigg|^m\bigg/|\lambda|^n\leq \varepsilon'.\]
Then
\begin{align*}
    |u(r_n(t))(a_m(t))+b_m(t)\Tilde{u}_{r_n(t)}(a_m(t),b_m(t))| \leq a_m(t)+b_m(t)=S_{m+1}|t|.
\end{align*}
Hence $\eqref{vanishing_model_a_m}$ is true for $m+1.$
By \eqref{vanishing_inequality_b_m}, it is clear that  \eqref{vanishing_model_b_m} is true for $m+1$.

\section{Marked points with proportional forward and backward measures}\label{sectionmarkedpoints}

Let $\Lambda$ be a Riemann surface and $\sigma\colon\Lambda \to \C^2$ be a marked point. When no confusion can arise, we use the same notation $\sigma$ to denote its graph function $\mathrm{id}\times \sigma\colon \Lambda\to\Lambda\times\C^2.$ Let $f:\Lambda\times\C^2 \to \Lambda\times\C^2$ be a holomorphic (not necessarily algebraic) family of regular plane polynomial automorphisms of degree $d\geq 2$ parameterized by $\Lambda,$ defined by $f(t,z)=(t,f_t(z))$.

In this section, we assume that the two Green measures $\mu_{f,\sigma}^+$ and $\mu_{f,\sigma}^-$ (recall the definitions~\ref{Greenmeasuremarkedpoint}) are non-vanishing and proportional. Thus, there exists a positive constant $\gamma >0$ and a harmonic function $H$ on $\Lambda$ such that
\begin{align}\label{proportional}
    G^+_t \circ \sigma(t) = \gamma G^-_t \circ \sigma(t) + H(t).
\end{align}

Let $t\in \Lambda$ be a parameter such that $\sigma(t)$ is $f_t$-periodic of period $k$. Let $u(t)$ and $s(t)$ be the two eigenvalues of the differential of $f_t^k$ at $\sigma(t)$. We say that $\sigma(t)$ (or just $t$) is 
\begin{enumerate}
    \item \emph{saddle} if $|u(t)| > 1 > |s(t)|$;
    \item \emph{semi-repelling} if $|u(t)|>1$ and $|s(t)|=1$;
    \item \emph{semi-attracting} if $|u(t)|=1$ and $|s(t)|<1$;
    \item \emph{repelling} if $|u(t)|>1$ and $|s(t)|>1$;
    \item \emph{attracting} if $|u(t)|<1$ and $|s(t)|<1$;
    \item \emph{neutral} otherwise.
\end{enumerate}
We will study periodic parameters $t$ based on the type of the multipliers, under the assumption~\eqref{proportional}.

\subsection{Saddle parameters}
\begin{prop}\label{saddle parameters}
    Let $f_t$ be a holomorphic family of regular plane polynomial automorphisms of degree $d\geq 2$ parameterized by a Riemann surface $\Lambda$ and $\sigma\colon\Lambda \to \C^2$ a marked point. Suppose that $\mu_{f,\sigma}^+\neq 0$ is proportional to $\mu^-_{f,\sigma}\neq 0$. If $\sigma(t_0)$ is a saddle periodic point, then the Jacobian $\Jac(f_{t_0})$ is a root of unity.
\end{prop}
There are two key ingredients in the proof. The first is the renormalization lemma \ref{saddleparameterslemma} for Green functions, which follows from Lemma \ref{renormalization saddle}. The second is the lower H\"{o}lder exponents of continuity of Green functions, an idea borrowed from Dujardin and Favre \cite{DFManinMumford2017}. In fact, their computations in \cite[Sect. 3]{DFManinMumford2017} are local, except for \cite[Lemma 3.3]{DFManinMumford2017}, which deals with algebraic curves. Nevertheless, we can replace \cite[Lemma 3.3]{DFManinMumford2017} with a local argument involving the construction of Pesin box (see, e.g., \cite[Sect. 4]{BLS93}), up to removing a subset of arbitrary small measure (see Lemma \ref{saddleparameterslemma2}). Another slight difference compared to \cite[Sect. 3]{DFManinMumford2017} is the possible higher-order tangency of the marked point, namely it can be the case that $p,q \geq 2$ in Lemma \ref{saddleparameterslemma} and the equation \eqref{Proportional saddle}.

\begin{proof}
Up to replacing $f$ by an iterate, we can suppose $\sigma(t_0)$ is fixed by $f_{t_0}$.

\textbf{Step 0: A renormalization lemma for Green functions.}
Denote by $\Dbb(r)$ the disk of radius 0 centered at the origin in the complex plane. Let $\sigma_0$ be the local analytic continuation of the saddle  point $\sigma(t_0)$ over an analytic open subset $U\subset \Lambda$ given by the implicit function theorem. Let $V\subset \C^2$ be a small open neighborhood of $\sigma(t_0)$. Let $\Psi$ be a biholomorphism such that, shrinking $V$ and $U$ if necessary, the following diagram commutes.
\[
\begin{tikzcd}
\mathbb{D}(\varepsilon)\times \mathbb{D}(1)^2 \arrow[d] \arrow[rr, "\psi"] &  & U\times V \arrow[d] \\
\mathbb{D}(\varepsilon) \arrow[rr, "\phi"']                                 &  & U                  
\end{tikzcd}
\]
The two vertical morphisms are the structural projections and $\varepsilon \ll 1.$ Moreover, the following is true. In the new coordinates, $\Dbb(1)^2$ is the product of the local unstable manifold and the local stable manifold above
any parameter $t\in \Dbb(\varepsilon)$. By abuse of notation, still denote by $\sigma(t)$ its conjugate $\psi^{-1}_t\circ \sigma\circ\phi(t)$, by $f_t$ its conjugate $\psi^{-1}_t \circ f_t\circ \psi_t$ and by $G^\pm_t$ its 
conjugate $G^\pm_t\circ \psi_t$. Denote by $G^\pm_0= G^\pm_{t_0}$. We can write
\begin{align*}
    \sigma(t)&=(\alpha(t),\beta(t)) \\
    f_t(x,y) &= ( u(t)x + xyu_t'(x,y), s(t)y +xys_t'(x,y)),
\end{align*}
where $\alpha(t)=t^q + \mathrm{h.o.t.}$, $\beta(t)=t^p +\mathrm{h.o.t.}$ with $p,q\in \Z_{\geq 1}$ and $u(t)$ (resp. $s(t)$)
the unstable (resp. stable) multiplier of $f_t$ at $\sigma(t)$. By Lemma \ref{renormalization saddle} we have

\begin{lem}[Renormalization of Green functions: saddle case]\label{saddleparameterslemma}
    There exist a positive integer $q\geq 1$ and a parameterization $\rho^u_0\colon \Dbb(\varepsilon) \to W^u_{\mathrm{loc}}(0)$ of the local unstable manifold of $\sigma(t_0)$ such that the following is true. Let $\lambda_u$ be a $q$-th root of $u(0)$ and $r^u_n(t) \coloneqq \frac{t}{\lambda_u^n}$. Then
    \begin{align*}
        G^+_{0} \circ \rho^u_0 (t^q) = \lim_{n\to +\infty} G^+_{r^u_n(t)}(f_{r^u_n(t)}^n(\sigma({r^u_n(t)})))= \lim_{n\to +\infty} d^nG^+_{r^u_n(t)}(\sigma({r^u_n(t)})),
    \end{align*}
    and the convergence is uniform.

    There exist positive integer $p\geq 1$ and a parameterization $\rho^s_0\colon \Dbb(\varepsilon) \to W^s_{\mathrm{loc}}(0)$ of the local stable manifold of $\sigma(0)$ such that the following is true. Let $\lambda_s$ be a $p-$th root of $s(0)$ and $r^s_n(t) \coloneqq \lambda_s^n t$. Then
    \begin{align*}
        G^-_{0} \circ \rho^s_0 (t^p) = \lim_{n\to +\infty} G^-_{r^s_n(t)}(f_{r^s_n(t)}^{-n}(\sigma({r^s_n(t)})))= \lim_{n\to +\infty} d^nG^-_{r^s_n(t)}(\sigma(r^s_n(t))),
    \end{align*}
    and the convergence is uniform.

    If moreover $|\lambda_u|>|\lambda_s^{-1}|$, then $\lim_{n\to +\infty}G^-_{r^u_n(t)}(f^{-n}_{r^u_n(t)}(\sigma(r^u_n(t))))=0,$ and the convergence is uniform.
\end{lem}

\textbf{Step 1: Show $|\lambda_u|=|\lambda_s^{-1}|.$} Suppose by contradiction that $|\lambda_u| > |\lambda_s^{-1}|$. Then by Lemma \ref{saddleparameterslemma} and \eqref{proportional},
\begin{align*} 
     G^+_{0} \circ \rho^u_0 (t^q) &=\lim_{n\to +\infty} d^{n}G^+_{r^u_n(t)}(\sigma(r^u_n(t))) =\lim_{n\to +\infty} \gamma d^{n}G^-_{r^u_n(t)}(\sigma(r^u_n(t))) + d^{n} H(r^u_n(t))\\
     &=\lim_{n\to +\infty} \gamma G^-_{r^u_n(t)}(f^{-n}_{r^u_n(t)}(\sigma(r^u_n(t)))) + d^{n} H(r^u_n(t) = \lim_{n\to \infty} d^{n} H(r^u_n(t)).
\end{align*}
Hence $G^+_{0} \circ \rho^u_0 (t^q)$ is harmonic, whence the contradiction. By symmetry, $|\lambda_u|=|\lambda_s^{-1}|$, an thus there exists $ \theta \in [0,2)$ such that $\lambda_u = e^{i\theta \pi} \lambda_s^{-1}.$ 

\textbf{Step 2: Show that $\theta$ is rational.}
Suppose by contradiction that $\theta$ is irrational. Let $\xi$ be a complex number on the unit circle. There exists a subsequence $n_j$ of $\N$ such that $e^{i\theta \pi n_j} \to \xi^{-1/p}$. Again by Lemma \ref{saddleparameterslemma} and \eqref{proportional}, we have
\begin{align*}
    G^+_{0} \circ \rho^u_0 (t^q) &=\lim_{n_j\to +\infty} \gamma d^{n_j}G^-_{r^u_{n_j}(t)}(\sigma(r^u_{n_j}(t))) + d^{n_j} H(r^u_{n_j}(t))\\
    &=\lim_{n_j\to +\infty} \gamma d^{n_j}G^-_{r^s_{n_j}(e^{-i\theta \pi n_j} t)}(\sigma(r^s_{n_j}(e^{-i\theta \pi n_j} t))) + d^{n_j} H(r^u_{n_j}(t))\\
    &= \gamma G^-_{0} \circ \rho^s_0 (\xi t^p) + \lim_{n_j\to +\infty} d^{n_j} H(r^u_{n_j}(t)) .
\end{align*}
The measure $\ddc G^+_{0} \circ \rho^u_0 (t^q) $ is thus rotation-invariant in a neighborhood of the origin, which is impossible (see \cite[p. 3450]{DFManinMumford2017}).

\textbf{Step 3: Applying \cite[Proposition 3.1]{DFManinMumford2017} of Dujardin and Favre.} Write $\theta=k_1/k_2,$ where $k_1\in \Z$ and $k_2\in \N^*.$ We have
\begin{align*}
     G^+_{0} \circ \rho^u_0 (t^q) &=\lim_{n\to +\infty} \gamma d^{2k_2n}G^-_{r^u_{2k_2n}(t)}(\sigma(r^u_{2k_2n}(t))) + d^{2k_2n} H(r^u_{2k_2n}(t))\\
    &=\lim_{n \to +\infty} \gamma d^{2k_2 n}G^-_{r^s_{2k_2n}( t)}(\sigma(r^s_{2k_2n}(t))) + d^{2k_2n} H(r^s_{2k_2 n}(t)).
\end{align*}
Setting $\Tilde{H}(t)\coloneqq \lim_{n\to \infty} d^{2k_2 n} H(r^s_{2k_2 n}(t))$, which is harmonic, we obtain
\begin{align}\label{Proportional saddle}
    G^+_{0} \circ \rho^u_0 (t^q) = \gamma G^-_{0} \circ \rho^s_0 ( t^p) + \Tilde{H}(t).
\end{align}
Let $\phi_q \colon t \mapsto t^q$ and $\phi_p t \colon \mapsto t^p$. Set $\mu^\pm_0 \coloneqq dd^c G^{\pm}_{0}\circ\rho^{u/s}_0$. Denote by $\chi_0^{u/s}$ the Lyapunov exponents of $f_{t_0}$ relative to the measure $\mu_{f_{t_0}}$ of maximal entropy of $f_{t_0}$, i.e.
\begin{align*}
    \chi_0^{u/s}\coloneqq\lim_{n\to +\infty}\frac{\pm 1}{n}\int \log \lVert Df^{\pm n}_x\rVert d\mu_{f_{t_0}}(x).
\end{align*}

\begin{lem}\label{saddleparameterslemma2}
There exist subsets $C^\pm \subset \Dbb(\varepsilon)$ with 
$\left( \phi_{q/p}^{*} \mu^\pm_0\right) (\Dbb(\varepsilon) \setminus C^\pm)\ll 1,$
such that at any point $z^\pm \in C^\pm$, we have
\begin{align*}
    \liminf_{r\to 0}\frac{1}{\log r}\log\left( \sup_{d(z^\pm, z)\leq r,\, z\in \Dbb(\varepsilon)}G^\pm_{0} \circ \rho_0^{u/s}(z^{q/p})\right) = \frac{\log d}{\pm \chi_0^{u/s}}.
\end{align*}
\end{lem}
\begin{proof}[Proof of Lemma \ref{saddleparameterslemma2}]
By symmetry, it suffices to the unstable case.
Let $E$ be the set of regular points of $\mu_{f_{t_0}}$ given by Pesin's theory (see, e.g., the beginning of \cite[Section 3.2]{DFManinMumford2017}). Let $A\subset E$ be of full $\mu_{f_{0}}$-measure.
Shrinking $\varepsilon,$ there exists $\overline{A}\subset \Dbb(\varepsilon)$ such that $\mu_0^+(\Dbb(\varepsilon) \setminus \overline{A}) \ll 1$ and for any $\overline{z} \in \overline{A}$, there exists $z\in A$ such that $\rho^u_0(\overline{z}) \in W^s_{\mathrm{loc}}(z)$ and $W^s_{\mathrm{loc}}(z)$ intersects $W^u_{\mathrm{loc}}(\sigma_0(t_0))$ transversely at $\rho^u_0(\overline{z})$ (see \cite[Sect. 4]{BLS93}). Now applying \cite[Proposition 3.1]{DFManinMumford2017} by replacing \cite[Lemma 3.3]{DFManinMumford2017} by the above property, we deduce the existence of $B\subset \Dbb(\varepsilon)$ with $\mu^+_0(\Dbb(\varepsilon)\setminus B)\ll 1$, such that at any point $z \in B$, the lower H\"older exponent of continuity of $G^+_{0} \circ \rho_0^{u}$ at $p$ are $\log d / \chi_0^u$. The lemma is then proved by remarking that $\phi_q \colon \mapsto t^q$ is a local isomorphism except at the origin. 
\end{proof}

\textbf{Step 4: Show that $|\Jac(f_{0})|=1$.}
It will imply $|u(0)|=|s(0)|^{-1}$. Since $|\lambda_u|=|\lambda_s^{-1}|$, we have $p=q$, and finally that $\Jac(f_{0}) = \lambda_u^p\lambda_s^p = e^{i\theta \pi p}$ is a root of unity.

We follow \cite[Sect. 3.1]{DFManinMumford2017}
Recall \cite[Sect. 3]{BS3}) that, we have $|\Jac(f_{0})| = \exp (\chi^u +\chi^s)$.
If $\chi_0^{u/s} = \pm \log d$, then $|\Jac(f_{0})|=1$. Suppose $\chi_0^s < -\log d.$ By \eqref{Proportional saddle}, $C^+ \cap C^- \neq \emptyset.$ Choose any point $z^0\in C^+ \cap C^-.$ Pick $z_n \in \Dbb(\varepsilon)$ such that $l_n\coloneqq|z_n - z^0| \to 0$ and $\log (G^- \circ \rho_0^s (z_n^p))/\log l_n \to \log d/(-\chi_0^s).$ Since $\Tilde{H}(z^0)=0$ and $\Tilde{H}$ is smooth, we have $\Tilde{H}(z_n) = O(l_n).$ Let $\varepsilon'$ be small enough so that $(1+\varepsilon') \log d/ (-\chi_0^s) < 1$. There exists a constant $c>0$ such that
\begin{align*}
    \gamma G_{0}^-\circ \rho^s_0(z_n^p) + \Tilde{H}(z_n) \geq \gamma l_n^{(1+\varepsilon') \log d/(- \chi_0^s}) + O(l_n) \geq c \, l_n^{(1+\epsilon') \log d/ -\chi_0^s}.
\end{align*}
On the other hand, we have
\begin{align*}
    \liminf_{n\to +\infty} \frac{\log G_{0}^+ \circ \rho^u_0(z^q_n)}{\log l_n} \geq \liminf_{n\to +\infty} \frac{1}{\log l_n} \log \left(\sup_{|z^0-z|\leq l_n}  G_{0}^+ \circ \rho^u_0(z^q) \right) \geq \log d/ \chi_0^u .
\end{align*}
Thus using \eqref{Proportional saddle} to combine the above two chains of inequalities and letting $\varepsilon' \to 0$, we obtain $\log d/ (- \chi_0^s) \geq \log d/ \chi_o^u$.
Now since $\frac{\log d}{ \chi_0^u} < 1$, the same argument in the stable direction gives the opposite inequality and finally  $ - \chi_0^s = \chi_0^u$, which implies that $|\Jac(f_{0})| = 1.$
\end{proof}

\subsection{Non-saddle parameters.}
In this subsection we show that except for the neural ones, other types of multipliers are not possible.
\subsubsection{Semi-repelling/semi-attracting parameters}
\begin{prop}\label{propsemi}
    Let $f_t$ be a holomorphic family of regular plane polynomial automorphisms of degree $d\geq 2$ parameterized by a Riemann surface $\Lambda$ and $\sigma\colon\Lambda \to \C^2$ a marked point. Suppose $\mu_{f,\sigma}^+\neq 0$ is proportional to $\mu^-_{f,\sigma}$. Then there exists no parameter $t_0$ such that $\sigma(t_0)$ is semi-repelling or semi-attracting.
\end{prop}
\begin{proof}
Suppose that $\sigma(t_0)$ is a semi-repelling fixed point. If the central multiplier is 1,  the implicit function theorem is no longer applicable for tracking this point. To address this, we can take an irreducible component $\Lambda'$ (and normalize it) of $\{(t,z)\in \Lambda \times \C^2 \ |\ f_t(z)=z \}$ which contains $\sigma(t_0)$. We then pull back the family $f_t$ via the base change $\Lambda' \to \Lambda$. Note that the base change preserves multipliers, allowing us to still assume that there is a local analytic continuation of the semi-repelling fixed point $\sigma(t_0)$. The following lemma holds, similarly to the saddle parameter case Lemma \ref{saddleparameterslemma}.
\begin{lem}[Renormalization of Green functions: semi-repelling case]\label{semiparameterslemma}
    There exist a positive integer $D \geq 1$, a complex number $\lambda$ outside the closed unit disk and a parameterization $\rho^u_0\colon \Dbb(\varepsilon) \to W^u_{\mathrm{loc}}(\sigma(t_0))$, such that, up to taking an iterate of $f_t,$ the following is true. Define $r_n(t) \coloneqq \frac{t}{\lambda^n}$, there exists a non constant holomorphic function $h\colon\Dbb(\varepsilon) \to \C$ such that
    \begin{align*}
    G^+_{0} \circ \rho^u_0(h(t)) &= \lim_{n\to +\infty} G^+_{r_n(t)} (f^n_{r_n(t)}(\sigma(r_n(t))))= \lim_{n\to +\infty} d^n G^+_{r_n(t)} (\sigma(r_n(t))).
    \end{align*}
    and $ \lim_{n\to + \infty}G^-_{r_n(t)}(f^{-n}_{r_n(t)}(\sigma(r_n(t))))=0$, and the convergence is uniform.
\end{lem}

By Lemma \ref{semiparameterslemma} and \eqref{proportional},
\begin{align*}
    G^+_{0} \circ \rho^u_0(h(t)) &=\lim_{n\to \infty}\left( \gamma G^-_{r_n(t)}(f^{-n}_{r_n(t)}\sigma(r_n(t))) + d^nH(r_n(t))  \right) =\lim_{n\to \infty} d^nH(r_n(t)).
\end{align*}
is harmonic,which is a contradiction. By symmetry, the semi-attracting case is not possible either.
\end{proof}

\subsubsection{Repelling/attracting parameters}
\begin{prop}\label{proprepelling}
    Let $f_t$ be a holomorphic family of regular plane polynomial automorphisms of degree $d\geq 2$ parameterized by a Riemann surface $\Lambda$ and $\sigma\colon\Lambda \to \C^2$ a marked point. Suppose $\mu_{f,\sigma}^+\neq 0$ is proportional to $\mu^-_{f,\sigma}$. Then there exists no parameter $t_0$ such that $\sigma(t_0)$ is repelling or attracting.
\end{prop}

\begin{proof}
Suppose by contradiction $\sigma(t_0)$ is a repelling fixed point. By symmetry, it suffices to consider this case. We argue locally (see the proof of Lemma \ref{saddleparameterslemma}) and assume $f_t$ can be expressed as:
\begin{align*}
    f_t(x,y) &= ( u(t)x, \Tilde{s}_t(x,y)).
\end{align*}
where $u(t)$ is an eigenvalue of $D_{\sigma(t)}f_t$. Since $|u(0)|>1$, the family $f^n_t(\sigma(t))$ can not be normal, implying that $t_0$ lies in the support of the forward Green measure. However, $t_0$ is an attracting fixed parameter for $f_t^{-1}$, thus locally uniformly $f^{-n}_t(\sigma(t))$ converge to $\sigma(t)$. Hence the family~$f^{-n}_t(\sigma(t))$ is normal at $t_0$, and $t_0$ does not belong to the support of the backward Green measure. Since these two measures are proportional, this leads to a contradiction.
\end{proof}

\section{Non-degenerate families of curves}\label{sectioncurves}
\subsection{Notations}\label{notations}
Let $\Lambda$ be a smooth quasi-projective curve with a smooth compactification $\B.$ Let~$f_t$ be an algebraic family of regular plane polynomial automorphisms of degree $d\geq 2$ parameterized by $\Lambda$. Define $\Fscr_n \colon \Lambda \times \A^2 \to \Lambda \times \A^2 \times \A^2$ by $\Fscr_n(t,z)=(t,f_t^n(z),f^{-n}_t(z))$.
It can be extended to a regular morphism $\Lambda \times \P^2 \to \Lambda \times \P^4$ that will still be denoted by $\Fscr_n$ , see~\cite[Lemma 6.1]{Lee2013}. Let $\Cscr\subset \Lambda\times \A^2$ be a family of curves parameterized by $\Lambda$. For any integer $n\geq 1,$ denote by $\Cscr_n$ the Zariski closure of $\Fscr_n(\Cscr)$ in $\B\times \P^4$.
Denote by $\pi$ (resp. $\pi'$) the projection $\B\times \P^2 \to \B$ (resp. $\B\times \P^4 \to \B$). The two projections $\Lambda\times \A^2\times\A^2 \to \Lambda\times \A^2$ will be denote by $p$ and $q$. The line bundle $\L$ (resp. $\L'$) is defined to be the pullback of $\O_{\P^2}(1)$ (resp. $\O_{\P^4}(1)$) by the projection $\B\times \P^4 \to \P^2$ (resp. $\B\times \P^4 \to \P^4$).

Recall that we use the additive notation for the tensor product of line bundles. If we have a morphism $ X\to B$ and a subvariety $Y\subset X,$ the subvariety $Y^{[2]}$ is defined to be the variety $Y \times_B Y$. If we have a line bundle $L$ on $X$ and we denote by $q_1,q_2 \colon Y^{[2]} \to Y$ the two projections, then the line bundle $L^{[2]}$ is defined to be $q_1^*L+q_2^*L$. If $h\colon Y \to \R$ is any function, then $h^{[2]}\coloneqq h\circ q_1 + h\circ q_2$. 

\subsection{Geometric canonical height functions}
Recall that $G_f\coloneqq \max\{G_f^+, G_f^-\}$ and the Green measure of $\Cscr$ is $ \mu_{f,\Cscr}\coloneqq \left( \ddc G_f \right)^2 \wedge [\Cscr]$.
We define the \emph{(geometric canonical) height} $\Tilde{h}_f(\Cscr)$ of $\Cscr$ (relative to $f$) to be the mass of the Green measure $\mu_{f,\Cscr}$
\begin{align}\label{DefHeightCurve}
    \Tilde{h}_f(\Cscr)\coloneqq\int_{\Lambda \times \C^2} \mu_{f,\Cscr}.
\end{align}
We say that $\Cscr$ is a \emph{non-degenerate} if its height is non-vanishing.
We can reinterpret the canonical height as the limit of the following intersection numbers.
\begin{prop}\label{hieghtIntersection}
    We have $\Tilde{h}_f(\Cscr) = \lim_n \frac{1}{d^{2n}} \Cscr_n \cdot {\L'}^2.$
\end{prop}

\begin{proof}
    Let $\omega'$ be the pull back by the projection $\B\times \P^4_\C \to \P^4_\C$ of the Chern form of the canonical line bundle $\O_{\P^4_\C}(1)$ endowed with the standard continuous metric (Example \ref{standradmetric}). Define the function $G_n:\Lambda\times\C^2\times\C^2 \to \R_+$ by $G_n(t,x,y)\coloneqq\frac{1}{d^n}\log^+\lVert (f_t^n(x,y),f_t^{-n}(x,y))\rVert$, where $\lVert \cdot \rVert$ is the maximum of the modulus of the coordinates, so that $\ddc G_n= \frac{1}{d^n}(\Fscr_n)^* \omega'$ on $\Lambda \times \C^2.$ Choose a very ample divisor such that we have a closed immersion $\imath\colon \B \hookrightarrow \P^N_\C$ for some integer $N\geq 1$ and $\Lambda = \imath^{-1}(\C^N)$. Denote by $T_f$ the trivial extension of $\ddc G_f$ to $\Lambda\times \P_\C^2.$ Let $\varphi_n$ be the continuous function of $\Lambda\times \P^2_\C$ such that $\varphi_n = G_n-G_f$ on $\Lambda\times \C^2$, so that $\frac{1}{d^n}(\Fscr_n)^* \omega' - T_f = dd^c \phi_n.$ We claim that there exists a constant $C>0$ such that
    \begin{align}\label{degeneration estimate}
        |\phi_n(t,z)| \leq \frac{C}{d^n}(\log^+|t| +1 )
    \end{align}
    for all $(t,z)\in \Lambda\times \P_\C^2.$ In fact, since $\B \setminus \Lambda$ is finite, we can argue locally by centering the local coordinate around each point of $\B\setminus\Lambda$ and suppose our family is holomorphic on the punctured disk. Then we apply \cite[Proposition 3.2]{irokawa2023hybrid} (the result is stated for H\'enon maps but the same proof also works for generalized H\'enon maps). 

    Following \cite{gauthier2020geometric}, we define $\Psi_r(t)\coloneqq\frac{1}{r}(\log\max(|t|,e^{2r})- \log\max(|t|,e^r))$.
    Observe that this function takes values in $[0,1]$ and equals to 1 if $|\lambda|\leq \exp(r)$ and 0 if $|\lambda| \geq  \exp(2r).$ The positive closed current $T_r \coloneqq \ddc(\log\max(|t|,e^r)))$ has finite mass independent of the radius $r>0.$ This function is in fact an example of DSH functions introduced by Dinh and Sibony (see e.g., \cite{DSLecture}). Since we have $\Cscr_n \cdot {\L'}^2=\int_{\Lambda \times \P^2_\C} \frac{1}{d^{2n}}(\Fscr_n)^* {\omega'}^2,$ it suffices to estimate the following
\begin{align*}
    I_{n,r} \coloneqq \left |\left \langle \left ( \frac{1}{d^{2n}} {(\Fscr_{n})}^{*}{\omega'}^2- T_f^2\right ) \wedge [\Cscr], \Psi_r\circ \pi \right \rangle \right | = \left|\left \langle \phi_n \left(\frac{1}{d^n}{(\Fscr_{n})}^{*}\omega' + T_f\right) \wedge [\Cscr], dd^c\left( \Psi_r \circ \pi \right ) \right \rangle \right |
\end{align*}
By B\'ezout (see, e.g., \cite{Okainequality}),
\begin{align*}
    I_{n,r}\leq \frac{C(2r+1)}{rd^n}\left(\lVert \frac{1}{d^n}{(\Fscr_{n})}^{*}{\omega'} \rVert_{\Lambda\times\P^2_\C} +\lVert T_f \rVert_{\Lambda\times\P^2_\C} \right)\lVert \Cscr \rVert_{\Lambda\times\P^2_\C} \int_{\Lambda \cap \overline{\Dbb}(0,e^{2r})} T_{2r} + T_r.
\end{align*}
By \eqref{degeneration estimate} and that the mass of $\frac{1}{d^n}{(\Fscr_{n})}^{*}{\omega'}$ is independent of $n$, there exists a constant $C'>0$ such that $I_{n,r}\leq C'/d^n$.
Letting $r\to +\infty$, $\left|\frac{1}{d^{2n}}\Cscr_n \cdot {\L'}^2 - \Tilde{h}_f(\Cscr)\right|\leq C'/d^n$. Now let $n\to +\infty$ to conclude.
 \end{proof}

\subsection{Positivity of non-degenerate families of curves}
Let us denote by $p,q$ be the two projections $\Lambda\times (\C^2)^2 \to \Lambda\times\C^2.$ If $f\colon \Lambda \times \C^2 \to \R$ is any function, then set $f^{[2]}\coloneqq f\circ p + f\circ q$. If $\Cscr$ is any family of curves, then denote by $\Cscr^{[2]}=\Cscr\times_\Lambda\Cscr.$
The aim of this subsection is to prove Proposition~\ref{positivityI_f}, which plays a crucial role in establishing a contradiction to the equidistribution theorem~\ref{equidistribution curves}.
\begin{prop}\label{positivityI_f}
    Suppose $\Cscr$ is a non-degenerate family of curves. Then there exists an integer $n$ such that
    \begin{align*}
        I_{f,f^n(\Cscr)}\coloneqq \int_{\Lambda\times \C^4} G_{f^n}^{[2]} \left( \ddc G_{f^n}^{[2]} \right)^3 \wedge [f^n(\Cscr)^{[2]}]>0.
    \end{align*}
\end{prop}
Set $G_{f,\Cscr}(t)\coloneqq \int_{\Cscr_t}G_{f_t} \ddc G_{f_t}$.
\begin{lem}\label{pushdownmeasure}
    We have the equality of measures $\pi_{*} \mu_{f,\Cscr}= \ddc G_{f,\Cscr}.$
\end{lem}
\begin{proof}
    Recall that $T_f$ is the trivial extension of $\ddc G_f$ on $\Lambda \times \P_\C^2$. It has continuous potential by construction and there exists a positive constant $C_{T_f}$ (namely, the mass of $T_f$) such that, on $\Lambda \times \P^2_\C$, we have $\ddc G_f + C_{T_f} [\Lambda\times L_\infty] = T_f,$ where $L_\infty$ denotes the line at infinity.
    Let $\varphi$ be a test function on $\Lambda$. Then we have
    \begin{align*}
        \langle\pi_*\mu_{f,\Cscr}, \varphi \rangle &=\int_{\Lambda \times \C^2} \pi^*\varphi\, \mu_{f,\Cscr} = \int_{\Lambda \times \P^2}\pi^*\varphi\, T_f^2 \wedge [\Cscr]\\
        &=\int_{\Lambda \times \P^2}\pi^*\varphi\, \left(\ddc G_f + C_{T_f} [\Lambda\times L_\infty]\right)\wedge T_f \wedge [\Cscr].
    \end{align*}
    Let $\Y\subset \Lambda \times L_\infty$ be a horizontal curve lying at infinity. Since $L_\infty$ is contracted to the same point at infinity by $f_t$, we have $\Fscr_n(\Y)=\Fscr_1(\Y)$ for all $n\geq 1.$ Denote by $\overline{\Fscr_n(\Y)}$ the Zariski closure of the set $\Fscr_n(\Y)$ in $\B\times \P^4.$ Then $\lim_{n\to \infty}\frac{1}{d^n}\overline{\Fscr_n(\Y)}\cdot \L'=0$. It implies $T_f\wedge [\Y]=0$. Thus
    \begin{align*}
        \langle\pi_*\mu_{f,\Cscr}, \varphi \rangle &=\int_{\Lambda \times \P^2}\pi^*\varphi\, \ddc G_f \wedge T_f \wedge [\Cscr] =\int_{\Lambda \times \P^2}  \pi^*\ddc\varphi \wedge  \left(G_f T \wedge [\Cscr]\right)\\
        &=\int_\Lambda \ddc\varphi\, G_{f,\Cscr} =\int_\Lambda \varphi \ddc G_{f,\Cscr}
    \end{align*}
\end{proof}

\begin{proof}[Proof of Proposition \ref{positivityI_f}]
Since $G_f (\ddc G_f)^2\wedge [\Cscr] = G_f \ddc G_f^+ \wedge \ddc G_f^- \wedge [\Cscr]= 0,$
we have
\begin{align*}
    G_{f}^{[2]}\left( \ddc G_{f}^{[2]} \right)^3 \wedge [\Cscr^{[2]}] = 3p^* \mu_{f,\Cscr}\wedge q^*\left( G_f \ddc G_f \wedge [\Cscr] \right)
    +3q^* \mu_{f,\Cscr}\wedge p^*\left( G_f \ddc G_f \wedge [\Cscr] \right)
\end{align*}

Since
\begin{align*}
    \int_{\Lambda\times (\C^2)^2} p^* \mu_\Cscr \wedge q^*\left( G_f \ddc G_f \wedge [\Cscr] \right) = \int_{\Lambda\times 
    \C^2} G_{f,\Cscr}\circ \pi\, \mu_{f,\Cscr}.
\end{align*}
Lemma \ref{pushdownmeasure} implies 
\begin{align*}
    I_{f,\Cscr}=6\int_{\Lambda} G_{f,\Cscr} \ddc G_{f,\Cscr}.
\end{align*}
Since $\Cscr$ is non-degenerate, Lemma \ref{pushdownmeasure} implies once again that there exists $t_0\in \supp(\ddc G_{f,\Cscr})$. By \cite{BS1}, up to taking an iterate of $\Cscr$, we may assume that $\Cscr_{t_0}$ intersects with $L_\infty$ only at the point $[0:1:0]$. Let $K_{t_0}$ be the zero locus of $G_{t_0}$ and $\Omega$ a connected component of $\Cscr_{t_0}\setminus K_{t_0}$. By the maximum principle, $\Omega$ is unbounded and $\Omega\cap L_\infty = [0:1:0].$ This implies that there are only finitely many connected components $\Omega_i, 1 \leq i\leq k$, of $\Cscr_{t_0}\setminus K_{t_0}$.
Denote by $\alpha_i \colon \Omega_i \to [0,+\infty]$ the continuous function $G^+_{t_0}/G^-_{t_0}.$
\begin{claim}\label{rangeinfty}
    There exists an $\alpha_i$ such that $\inf_{z\in \Omega_i}\alpha_i(z) = 0$.
\end{claim}
Suppose $\alpha_i$ satisfies the claim. Denote by $C\coloneqq \sup_{z\in \Omega_i} \alpha_i$. Take $N\in \N$ large enough so that $d^{2N}C \gg 1.$ Then the component $f^N(\Omega_i)$ intersects properly the hypersurface $\{G^+_{t_0} = G^-_{t_0}\neq 0 \}$ which is a subset of $\supp(\ddc {G_{t_0}}|_{f^N(\Cscr_{t_0})}).$ Hence $G_{f,f^N(\Cscr)}(t_0)>0$. By continuity of $G_{f,f^N(\Cscr)},$ we know that $I_{f,f^N(\Cscr)}>0.$
\end{proof}
\begin{proof}[Proof of Claim \ref{rangeinfty}]
    Since there are only finitely many components, if Claim~\ref{rangeinfty} is false, we can assume there exist constants $0< C_1\leq C_2 \leq +\infty$ such that the range of $\alpha_i$ is contained in $[C_1,C_2]$ for all~$i.$ For any positive integer $N$, the range of $\alpha_i \circ f^N$ is then contained in $[d^{2N}C_1,d^{2N}C_2]$. Now, take $N$ large enough so that $d^{2N}C_1 \gg 1.$ Since $G^{\pm}_f$ is continuous, there exists a small open neighborhood $U\subset \Lambda$ of $t_0,$ such that $\inf\left( G^+_t/G^-_t\right) \gg 1$ for all $t\in U$, where the infimum is taken over $f^N(\Cscr_t)\setminus K_t$. In particular, $G_{f,f^N(\Cscr)}(t) = 0$ for all $t\in U$, which implies $\ddc G_{f,f^N(\Cscr)}=0$ on $U.$ On the other hand, we have the following:
    \begin{align*}
        \int_{U\times \C^2}(\ddc G_f)^2\wedge [f^n(\Cscr)] =\int_{U \times \C^2}f^*(\ddc G_f)^2\wedge [\Cscr] = \int_{U \times \C^2}(\ddc G_f)^2\wedge [\Cscr].
    \end{align*}
    Thus, by Lemma \ref{pushdownmeasure}, $\ddc G_{f,\Cscr}=0$ on $U$, implying that $t_0 \notin \supp(dd^c G_{f,\Cscr})$.
\end{proof}

\subsection{Dissipative families of quadratic H\'enon maps}
In this subsection, we focus on families of dissipative quadratic H\'enon maps. Namely, we consider the family
\begin{align*}
    f\colon(t,x,y)\in\C\times \C^2 \mapsto (t,y,y^2+t-\delta x)\in \C\times \C^2,
\end{align*}
where $\delta$ is a complex number in the open unit disk. We denote by
\begin{align*}
    y^\pm_{\delta,t}\coloneqq\frac{(1+\delta) \pm \sqrt{(1+\delta)^2 - 4t}}{2}
\end{align*}
the two complex roots of the equation $y^2+t=(1+\delta)y$, so that $(y^+_{\delta,t},y^+_{\delta,t})$  and $(y^-_{\delta,t},y^-_{\delta,t})$ are the two fixed points of $f_t$, and $|y^\pm_{\delta,t}| \to \infty$ if $|t|\to \infty.$ 

Recall that $\Dbb(z,r)$ denotes the open disk of center $z\in \C$ and radius $r.$ If $z=(z_1,\cdots,z_n)\in \C^n,$ then $\Dbb(z,r)$ denotes the polydisk $\prod_i\Dbb(z_i,r)$. 
\begin{lem}\label{Juliaquadratic}
    For $|t|\gg 1$ big enough, $K_t$ is contained in the four bidisks $\Dbb((y^\pm_{\delta,t},y^\pm_{\delta,t}),2)$. 
\end{lem}

Benedetto \cite[Lemma 5.1. and Lemma 6.1]{Benedetto2007} provided estimates for the size of the filled Julia set of one variable polynomial, while Ingram \cite[Lemma 3.1]{Ingram2011CanonicalHF} estimated the size of the filled Julia set for non archimedean H\'enon maps of Jacobian -1. The proof of Lemma~\ref{Juliaquadratic} follows the approach of Benedetto.

\begin{proof}
    Denote by $K_{1,t}$ (resp. $K_{2,t}$) the image of the filled Julia set $K_t$ by the first (resp. second) projection of $\C^2 \to \C.$ By the invariance of the filled Julia set, we have $K_{1,t}=K_{2,t}$ for all $t$.
    Let $\overline{\Dbb}(a,r)$ be the smallest closed disk containing $K_{1,t}$. Then $\overline{\Dbb}((a,a),r)$ is the smallest closed polydisk containing $K_t.$ Denote by $b_1$ and $b_2$ the two roots of the equation $y^2+t=(1+\delta)a$ and $s\coloneqq |b_1+b_2|/2$.

    We first show that the filled Julia set $K_t$ is contained in the four polydisks $\Dbb((b_i,b_j),2)$, where $i,j=1,2$. It suffices by the above to prove that
    $K_{1,t}\subset \Dbb(b_1,1) \cup \Dbb(b_2,1)$.
     For any $(x,y)\in K_t,$ by the invariance of the filled Julia set, we have
    \begin{align}\label{filled julia}
         r \geq |y^2+t-\delta x-a| = |y^2+t-(1+\delta)a - \delta(x - a)|\geq |y-b_1||y-b_2|-|\delta|r.
    \end{align}
     By the minimality of $\overline{\Dbb}(a,r)$, $K_{1,t}\setminus \overline{\Dbb}((b_1+b_2)/2,r) \neq \emptyset.$ Thus there exist $y\in K_{1,t}$ and some $i=1$ or 2, say $i=1$, such that $|y - b_1| \geq |y-(b_1+b_2)/2| \geq r.$ By triangle inequality, $|y-b_2|\geq|y-b_1|-|b_1-b_2|\geq r- s.$ By \eqref{filled julia}, if we suppose $r\geq s$, then $r \geq r(r-s)-|\delta| r.$ We deduce $s\geq r-|\delta|-1>r-2,$ which trivially remains true if $s\geq r.$ 

    Suppose by contradiction that there exists $y\in K_{1,\delta}\setminus\left(\Dbb(b_1,1) \cup \Dbb(b_2,1)\right)$. Since $|t|\gg 1$, we have $r\geq 2|y^\pm_{\delta,t}|\gg 1$ and $s\gg 1$. Thus the two disks $\Dbb(b_1,1)$ and $\Dbb(b_2,1)$ are disjoint and $\min_z|z-b_1||z-b_2|\geq 2s-1$, where the minimum is taken from all $z\in K_{1,\delta}\setminus\left(\Dbb(b_1,1) \cup \Dbb(b_2,1)\right).$ Still by \eqref{filled julia}, we have $r\geq 2s-1 -|\delta| r \geq 2(r-2)-1-|\delta|r$. Thus $r\leq 5/(1-|\delta|)$ and $r$ is bounded, which is a contradiction.
    Recall that, up to reordering $y^+_{\delta,t}$ and $y^-_{\delta,t}$, we have $y^{+/-}_{\delta,t}\in \Dbb(b_{1/2},1)$. Hence replacing $\Dbb(b_{1/2},1)$ by $\Dbb(y^{+/-}_{\delta,t},2)$, the filled Julia set $K_t$ is contained in the 4 disks $\Dbb((y^\pm_{\delta,t},y^\pm_{\delta,t}),2)$. 
\end{proof}

\begin{prop}\label{DegenerationFilledJulia}
    For $|t|\gg 1$ large enough, $K_t$ is contained in the eight bidisks of centers a point of $\Sigma_t$ (recall~\eqref{Sigma_t}) and of radius $r_t>0$ with $\lim_{|t|\to +\infty} r_t =0.$
\end{prop}
\begin{proof}
Fix $\xi^\pm_x, \xi_y^+\in \C$ such that $|\xi^\pm_x|, |\xi_y^+|<2$. For all $t$ with $|t|\gg 1,$ let $x^\pm_t\coloneqq y^\pm_{\delta,t}+\xi^\pm_x$ and $y_t^\pm \coloneqq y^+_{\delta,t}+\xi_y^\pm$. 

\textbf{Case 1: Analysis of} $K_t\cap \Dbb((y^-_{\delta,t},y^+_{\delta,t}),2)$.
The distance between $(y_t^+)^2+t-\delta x^-_t$ and the center $y^+_{\delta,t}$ is
\begin{align*}
    \Big\vert \left(y^+_{\delta,t}+\xi_y^+ \right)^2+t-\delta \left(y^-_{\delta,t}+\xi^-_x\right)- y^+_{\delta,t} \Big\vert
    &=\Big\vert \left(2 \xi_y^+ +\delta\right)y^+_{\delta,t} -\delta y^-_{\delta,t} +(\xi_y^+)^2 - \delta \xi^-_x \Big\vert\\
    &=\Big\vert(\xi_y^+ +\delta)\sqrt{(1+\delta)^2-4t}+O(1)\Big\vert.
\end{align*}
The distance between $(y_t^+)^2+t-\delta x^-_t$ and the center $y^-_{\delta,t}$ is 
\begin{align*}
    \Big\vert \left(y^+_{\delta,t}+\xi_y^+ \right)^2+t-\delta \left(y^-_{\delta,t}+\xi^-_x\right)- y^-_{\delta,t} \Big\vert
    &=\Big\vert (1+\delta)\left(y^+_{\delta,t} - y^-_{\delta,t}\right) +2\xi_y^+ y^+_{\delta,t} +(\xi_y^+)^2 -\delta \xi^-_x  \Big\vert\\
    &=\Big\vert(1+\delta+\xi_y^+)\sqrt{(1+\delta)^2-4t}+O(1) \Big\vert
\end{align*}
The distance between $1/\delta ((x_t^-)^2+t-y_t^+)$ and $y^+_{\delta,t}$ is
\begin{align*}
     \Big\vert1/\delta ((x_t^-)^2+t-y_t^+)- y^+_{\delta,t} \Big\vert &= \Big\vert 1/\delta(1+\delta+2\xi_x^-)y^-_{\delta,t} -1/\delta(1+\delta)y^+_{\delta,t}+ 1/\delta(2\xi_x^--\xi_y^+)\Big\vert\\
     &=\Big\vert-1/\delta(1+\delta+\xi_x^-)\sqrt{(1+\delta)^2-4t} +O(1) \Big\vert.
\end{align*}
The distance between $1/\delta ((x_t^-)^2+t-y_t^+)$ and $y^-_{\delta,t}$ is
\begin{align*}
    \Big\vert1/\delta ((x_t^-)^2+t-y_t^+)- y^-_{\delta,t} \Big\vert &= \Big\vert 1/\delta(1+2\xi_x^-)y^-_{\delta,t} -1/\delta y^+_{\delta,t}+ 1/\delta((\xi_x^-)^2-\xi_y^+)\Big\vert\\
    &=\Big\vert -1/\delta(1+\xi^-_x)\sqrt{(1+\delta)^2-4t}+O(1) \Big\vert.
\end{align*}
By analysing the coefficient of $\sqrt{(1+\delta)^2-4t}$ we observe that for any (small) radius $r>0$ there exists $T(r)\gg 1$ such that for any $|t|\geq T(r)$, if $z$ is not contained in the union of the 4 polydisks of radius $r$ and of center $(y^-_{\delta,t}+\alpha, y^+_{\delta,t}+\beta)$, where $\alpha\in \{-1-\delta,-\delta \}$ and $\beta \in \{-1,-1-\delta\},$ Then $z\notin K_t.$

To go further, We can iterate one more time $f^2(x,y)=(p_1(x,y),p_2(x,y))$
with $p_1(x,y) = y^2+t-\delta x$ and $p_2(x,y) = (y^2+t-\delta x)^2+t-\delta y.$
We have
\begin{align*}
    p_2(x^-_t,y_t^+)&=\Big((y^+_{\delta,t}+\xi^+_y)^2 +t -\delta(y^-_{\delta,t} +\xi^-_x)\Big)^2 +t -\delta(y^+_{\delta,t}+\xi^+_y)\\
    &=\Big((1+\delta+2\xi_y^+)y^+_{\delta,t} -\delta y^-_{\delta,t} + (\xi_y^+)^2-\delta \xi_x^- \Big)^2 +t-\delta(y^+_{\delta,t} +\xi_y^+) \\
    &=-\Big((1+\delta+2\xi_y^+)^2 +2\delta(1+\delta+2\xi_y^+)+\delta^2 -1\Big)t +o(t).
    \end{align*}
Recalling $y^+_{\delta,t}y^-_{\delta,t}=t,$ we have
\begin{align*}
    p_2(x^-_t,y_t^+)&=(1+\delta+2\xi_y^+)^2 ((1+\delta)y^+_{\delta,t} -t) +\delta^2((1+\delta)y^-_{\delta,t}-t)-2\delta(1+\delta+2\xi_y^+)t+t+o(t)\\
    &=-\Big((1+\delta+2\xi_y^+)^2 +2\delta(1+\delta+2\xi_y^+)+\delta^2 -1\Big)t +o(t)
\end{align*}
The coefficient of $t$ vanishes if and only if
$1+\delta+2\xi_y^+ = -1\pm \sqrt{2-\delta^2},$ i.e.
\begin{align*}
    \xi_y^+ = \frac{-2 \pm \sqrt{2-\delta^2} -\delta}{2}.
\end{align*}
If $\xi_y^+=-1$ or $\xi_y^+=-1-\delta,$ then $|\delta|=1$, which is not possible by our hypothesis of dissipativity. Let $r_0 >0$ be the minimum of the distance between $\frac{-2 \pm \sqrt{2-\delta^2} -\delta}{2}$ and $-1$ or $-1-\delta$. For any $|t| \geq T(r_0)$ and any $z$ in the union of the 4 polydisks of radius $r$ and of center $(y^-_{\delta,t}+\alpha, y^+_{\delta,t}+\beta)$, where $\alpha\in \{-1-\delta,-\delta \}$ and $\beta \in \{-1,-1-\delta\},$ since the coefficient of $t$ in $p_2(x^-_t,y^+_t)$ is non zero, possibly enlarging $t,$ we have $z\notin K_t.$ 

In conclusion, for $t$ large enough, 
\begin{align*}
    K_t\cap \Dbb((y^-_{\delta,t},y^+_{\delta,t}),2)=\emptyset.
\end{align*}

\textbf{Case 2: Analysis of} $K_t\cap \Dbb((y^+_{\delta,t},y^+_{\delta,t}),2)$.
Now we Analyse $K_t \cap\Dbb((y^+_{\delta,t}, y^+_{\delta,t}),2)$.
The distance between $(y_t^+)^2+t-\delta x^+_t$ and the center $y^+_{\delta,t}$ is
\begin{align*}
    \Big\vert \left(y^+_{\delta,t}+\xi_y^+ \right)^2+t-\delta \left(y^+_{\delta,t}+\xi^+_x\right)- y^+_{\delta,t} \Big\vert
    &=\Big\vert 2 \xi_y^+ y^+_{\delta,t} + (\xi_y^+)^2 -\delta \xi_x^+\Big\vert.
\end{align*}
    The distance between $(y_t^+)^2+t-\delta x^+_t$ and the center $y^-_{\delta,t}$ is
    \begin{align*}
        \Big\vert \left(y^+_{\delta,t}+\xi_y^+ \right)^2+t-\delta \left(y^+_{\delta,t}+\xi^+_x\right)- y^-_{\delta,t} \Big\vert
        &=\Big\vert y^+_{\delta,t} - y^-_{\delta,t}+2\xi_y^+ y^+_{\delta,t} + \xi^2_y - \delta \xi_x^+ \Big\vert\\
        &=\Big\vert (1+\xi_y^+)\sqrt{(1+\delta)^2-4t} + (1+\delta)\xi_y^+ +\xi^2_y-\delta \xi^+_x  \Big\vert
    \end{align*}
The distance between  $1/\delta ((x_t^+)^2+t-y_t^+)$ and $y^+_{\delta,t}$ is
\begin{align*}
    \Big\vert 1/\delta ((x_t^+)^2+t-y_t^+)- y^+_{\delta,t} \Big\vert &= \Big\vert (2\xi_x^+/\delta) y^+_{\delta,t} +1/\delta((\xi_x^-)^2-\xi_y^+) \Big\vert\\
    &=\Big\vert\xi_x^+/\delta\sqrt{(1+\delta)^2-4t}+O(1) \Big\vert.
\end{align*}
The distance between  $1/\delta ((x_t^+)^2+t-y_t^+)$ and $y^-_{\delta,t}$ is
\begin{align*}
    \Big\vert 1/\delta ((x_t^+)^2+t-y_t^+)- y^-_{\delta,t} \Big\vert &= \Big\vert (2\xi_x^+/\delta) y^+_{\delta,t}-1/\delta y^-_{\delta,t} +1/\delta((\xi_x^-)^2-\xi_y^+) \Big\vert\\
    &=\Big\vert (\xi_x^++\delta)\sqrt{(1+\delta)^2-4t}+O(1) \Big\vert.
\end{align*}
By analysing the coefficient of $\sqrt{(1+\delta)^2-4t}$ we observe that for any (small) radius $r>0$ there exists $T(r)\gg 1$ such that for any $|t|\geq T(r)$, if $z$ is not contained in the union of the 4 polydisks of radius $r$ and of center $(y^+_{\delta,t}+\alpha, y^+_{\delta,t}+\beta)$, where $\alpha\in \{0,-\delta \}$ and $\beta \in \{0,-1\},$ then $z\notin K_t.$ 

The second iteration will not give more restrictions in this case and we will not compute it here.\\
 
\textbf{Other cases.}
By symmetry, the other two cases are similar. More precisely, we have 
\begin{align*}
    K_t\cap \Dbb((y^+_{\delta,t},y^-_{\delta,t}),2) =\emptyset
\end{align*}
for any (small) radius $r>0$ there exists $T(r)\gg 1$ such that for any $|t|\geq T(r)$, if $z$ is not contained in the union of the 4 polydisks of radius $r$ and of center $(y^-_{\delta,t}+\alpha, y^-_{\delta,t}+\beta)$, where $\alpha\in \{0,-\delta \}$ and $\beta \in \{0,-1\},$ then $z\notin K_t.$ 
\end{proof}

\begin{prop}\label{comparaisonPrinciple}
    Let $\Cscr\subset \C\times \C^2$ be a family of curves parameterized by $\C$. Suppose there exists $r>1$ such that 
    \begin{enumerate}
        \item for all $t\in\C$ with $|t|\geq r$, $\Cscr_t \cap K_t =\emptyset;$
        \item there exists a sequence $z_n$ of points in $\Cscr$ such that $|z_n|<r$ for all $n$ and $\lim_{n\to +\infty} G_f(z_n) = 0.$ 
    \end{enumerate} 
    Then $\Cscr$ is non-degenerate.
\end{prop}
\begin{proof}
    Define $\Omega\coloneqq\Dbb(r)\cap \Cscr$, where $\Dbb(r)\subset \C^3$ is the polydisk of radius $r$ centered at the origin. Denote by $\partial\Omega$ the boundary of $\Omega$ with respect to the topology of $\Cscr.$ Since $\Cscr_t \cap K_t =\emptyset$ for all $|t|\geq r$, there exists $\varepsilon>0$ such that $G\geq 2\varepsilon$ on $\partial \Omega.$ Applying the singular version of comparison principle \cite[Theorem 4.3]{BedfordComparisonPrincipal} with $X=\Omega, u=G_f$ and $v=\varepsilon+ \varepsilon \lVert z \rVert/2r.$ Note that by the second point (2), the subset $\{G_f < \varepsilon+ \varepsilon \lVert z \rVert/2r\}\neq \emptyset.$  
    we obtain thus 
    \begin{align*} 
        \int_{\{G_f < \varepsilon+ \varepsilon \lVert z \rVert/2r\}} (\ddc(G_f))^2 &\geq \int_{\{G < \varepsilon+ \varepsilon \lVert z \rVert/2r\}} (\ddc(\varepsilon+ \varepsilon \lVert z \rVert/2r))^2\\
        &>\int_{\{G < \varepsilon+ \varepsilon \lVert z \rVert/2r\}} (\ddc(\varepsilon))^2 =0
    \end{align*}
\end{proof}

On the other hand, for a conservative family of quadratic H\'enon maps, we have examples of degenerate families of curves.
\begin{ex}\normalfont\label{degenerate curve}
    Let $f\colon \C \times \C^2 \to \C^2$ defined by $f_t(x,y) = (y,y^2+t+x)$. Then its inverse is $f^{-1}_t(x,y) = (-x^2-t+y,x)$ and we have $\tau \circ f^n_t \circ \tau = f^{-n}_t$. Consider the involution $\tau (x,y) = (-y,-x)$ and its curve of fixed points $C\coloneqq \{x+y=0\}$. 
    Since $G_{f_t}^+ \circ \tau = G_{f_t}^-$, we have  $G_{f,\C \times C} (t)=0$. Hence by Lemma \ref{pushdownmeasure}, the constant family of curves $\C\times C$ is degenerate.
\end{ex}

\section{General height inequalities for families of regular plane polynomial automorphisms}\label{SectHeightInequalities}
\subsection{Call-Silverman type height inequalities}
In this subsection only, the variety $\B$ can have any dimension. Recall that we have an algebraic family $f:\Lambda\times \A^2_\K \to \Lambda\times \A^2_\K$ of regular plane polynomial automorphisms defined over a number field $\K$, where $\Lambda$ is a Zariski open subset of $\B.$ The function $f_t$ is defined by $f(t,z)=(t,f_t(x))$.

In \cite{CallSilverman}, Call and Silverman proved results on variation of canonical heights for families of polarized endomorphisms. We establish similar results (Lemma~\ref{heightinequality1}) for families of regular plane polynomial automorphisms. A crucial part of the proof is the effectiveness of a divisor~\eqref{effective}, as shown by Kawaguchi in \cite{KawaguchiAffine}. Denote by $h$ (resp. $h'$) the standard Weil height function on $\P_\K^2$ (resp. $\P_\K^4$), see also  Example~\ref{standradmetric}. On each fiber $\A_\K^2$ over $t\in \Lambda(\Qbar)$, the canonical height functions are defined to be $\hat{h}^\pm_{f_t}(x)\coloneqq \lim_{n\to +\infty}\frac{1}{d^n}h(f_t^{\pm n}(x)),$ and $\hat{h}_{f_t}=\hat{h}^+_{f_t} + \hat{h}^-_{f_t}$. These canonical height functions are well-defined and non-negative. Moreover, a point $z\in \A^2(\Qbar)$ is $f_t$-periodic if and only if $\hat{h}^\pm_{f_t}(x)=0.$ We refer to~\cite{Kawaguchi13} for more details.

\begin{lem}\label{heightinequalitylemma1}
    Let $\M$ be a very ample divisor on $\B.$
    Up to shrinking $\Lambda$, there exist positive constants $C_1, C_2>0$ such that for all $(t,x)\in \Lambda\times \A^2(\Qbar),$ we have
    \begin{align}\label{heightinequality1}
        h(x)\leq \hat{h}_{f_t}(x)+C_1(h_\Mscr(t) +1)
    \end{align}
    and
    \begin{align}\label{heightinequality2}
        \hat{h}_{f_t}(x)\leq 2h(x)+ C_2(h_\Mscr(t) + 1).
    \end{align}
\end{lem}

\begin{proof}
    We first establish the height inequality \eqref{heightinequality1}.
    Viewing the family of maps as a single map over the function field $\mathbb{L}\coloneqq \Qbar(\B)$, Kawaguchi~\cite[Theorem 2.1]{KawaguchiAffine} showed that there exist a morphism $\overline{\pi}_2\colon Y \to \P_{\mathbb{L}}^2$, obtained by successive blow ups of points at infinity, along with two other morphisms $\overline{\pi}_1$ and $\overline{\pi}_3$, such that the following diagram commutes:
\begin{equation*}
\begin{tikzcd}
             & Y \arrow[d, "\overline{\pi}_2"] \arrow[rd, "\overline{\pi}_3"] \arrow[ld, "\overline{\pi}_1"']    &              \\
\mathbb{P}_\mathbb{L}^2 & \mathbb{P}_\mathbb{L}^2 \arrow[l, "f^{-1}"', dashed] \arrow[r, "f", dashed] & \mathbb{P}^2_\mathbb{L}
\end{tikzcd}    
\end{equation*}
Moreover, denoting by $H_\infty$ the line at infinity, the divisor 
\begin{align}\label{effective}
    D\coloneqq\overline{\pi}_3^{*}H_\infty + \overline{\pi}_1^{*}H_\infty - (d+\frac{1}{d})\overline{\pi}_2^{*}H_\infty
\end{align}
is effective.
Hence, up to shrinking $\Lambda$, there exists a smooth projective model $\pi_\Y \colon\Y \to \B$ of $Y$ (i.e. $\Y$ is smooth projective, $\pi_\Y$ is surjective, and the generic fiber of $\pi_\Y$ is isomorphic to $Y$), with three rational map $\Pi_i$, regular over $\Lambda$, and extending $\pi_i$.
Let $\Y'$ be the desingularization of the image of the map $(\Pi_1, \Pi_2, \Pi_3)\colon\Y \dashrightarrow \B\times(\P^2)^3$. Denote by $\Theta_i$ the $i-$th projection to $\B\times \P^2$ and let $\pi_{\Y'}\coloneqq \pi \circ \Theta_2$. Consider the divisor
\begin{align*}
    \D\coloneqq \Theta_3^{*}(B\times H_\infty) + \Theta_1^{*}(B\times H_\infty) - (d+\frac{1}{d})\Theta_2^{*}(B\times H_\infty).
\end{align*}
Up to shrinking $\Lambda$ and taking a large multiple of $\M$, $\D+\pi_{\Y'}^*\M$ is an effective divisor supported outside $\Theta_2^{-1}(\Lambda\times \A^2)$. Up to taking a finite field extension of $\K$, we can assume all the varieties and divisors above are defined over $\K$.
Since the Weil height associated to an effective divisor is bounded below outside of the base locus of its linear system, there exists a positive constant $C >0$ such that
\begin{align*}
    (1+\frac{1}{d^2})h(x) \leq \frac{1}{d}h(f_t(x)) + \frac{1}{d}h(f_t^{-1}(x))  +\frac{C}{d}(h_\M(t)+1)
\end{align*}
for all $(t,x)\in \Lambda\times \A^2(\Qbar)$.
By induction, for all positive integers $n\geq 1$, 
\begin{align*}
    (1+\frac{1}{d^{2^n}})h(x)\leq \frac{1}{d^{2^n}}h(f_t^{2^n}(x)) + \frac{1}{d^{2^n}}h(f_t^{-2^n}(x)) +a_n C(h_\M(t)+1),
\end{align*}
where $a_1 = 1/d$ and $a_{n+1}=(1+2/d^{2^{n-1}} + 1/d^{2^n})a_n$. Passing to the limit $n\to +\infty,$ we obtain finally that $h(x)\leq \hat{h}^+_{f_t}(x) + \hat{h}^-_{f_t}(x) + C_1(h_\M(t)+1),$ where $C_1 = C \lim_{n\to +\infty}a_n.$
\newline

We now prove the other height inequality \eqref{heightinequality2}.
Let $\X$ be the desingularization of the graph of the rational map $\Fscr_1\colon \B\times \P^2 \dashrightarrow \B\times \P^4$. We denote by $\Phi_1$ the projection $\X\to \B\times \P^2$ and by $\Gscr_1$ the projection $\X\to \B\times \P^4$. Denote by $\L\coloneqq\O_{\B \times \P^2}(1)$ and $\L'\coloneqq\O_{\B \times \P^4}(1)$. There exists a vertical divisor $\Vscr$ such that $\frac{1}{d}\Gscr_1^*\L' - \Phi_1^*\L = \O_\X(\Vscr),$ so that
\begin{align*}
    -m\O_\X(\pi_\X^* \Mscr) \leq \frac{1}{d} \Gscr_1^*\L' - \Phi_1^*\L \leq m\O_\X(\pi_\X^* \Mscr)
\end{align*}
for some large integer $m>0$. Up to 
shrinking $\Lambda,$ We can assume that $\supp(\Mscr) \subset \B\setminus\Lambda,$ so that for any $t\in \Lambda(\Qbar)$ and $z\in \Lambda \times \A^2(\Qbar)$ with $\pi(z)=t,$ we have, by Weil's height machine,
\begin{align*}
    |\frac{1}{d}h_{\L'}(\Fscr_1(z)) - h_{\L}(z)|\leq m h_{\Mscr}(t) + O(1).
\end{align*}
In particular, there exists a positive constant $C>0$ such that for all $(t,x)\in \Lambda\times \A^2(\Qbar),$ we have
\begin{align*}
    \frac{1}{d}h(f_t(x)) \leq h(x)+mh_\Mscr(t)+O(1).
\end{align*}
By induction, for all $n\in \N^*,$ we have
\begin{align}\label{inequality2}
    \frac{h(f_t^{n+1}(x))}{d^{n+1}} \leq \frac{h(f_t^n(x))}{d^n}+\frac{mh_\Mscr(t)}{d^n}+\frac{O(1)}{d^n}\leq h(x)+m(\sum_{i=0}^n \frac{1}{d^i})h_\Mscr(t)+(\sum_{i=0}^n \frac{1}{d^i})O(1).
\end{align}
Passing to the limit $n\to +\infty$, there exists a positive constant $C^+ >0$ such that
\begin{align*}
    \hat{h}_{f_t}^+(x) \leq h(x) + C^+(h_\Mscr(t) + 1).
\end{align*}
Similarly, iterating backwards, there exists a positive constant $C^- >0$ such that
\begin{align*}
    \hat{h}_{f_t}^+(x) \leq h(x) + C^-(h_\Mscr(t) + 1).
\end{align*}
Hence denoting by $C\coloneqq C^+ + C^-,$ we get
\begin{align*}
    \hat{h}_{f_t}(x)\leq 2h(x)+ 2C(h_\Mscr(t) + 1)
\end{align*}
for all $(t,x)\in \Lambda\times \A^2(\Qbar)$.
\end{proof}

\subsection{Height inequalities on non-degenerate families of curves}
\begin{lem}\label{height inequality curve}
    Let $\Cscr\subset \Lambda \times \A^2$ be a non-degenerate family of curves defined over a number field~$\K$. Let $\M$ be any ample $\Q$-line bundle on $\B$. Then there exist positive constants $C_3, C_4 >0$, a non-empty Zariski open subset $\Cscr^0$ of $\Cscr$ such that for all $(t,x)\in\Cscr^0(\Qbar)$, we have
    \begin{align*}
        h_{ \M}(t) \leq C_3 \hat{h}_{f_t}(x) +C_4.
    \end{align*}

\end{lem}
See \cite[Theorem 6.2.2]{YZadelic} (or \cite[Theorem 5.4]{gauthier2023sparsity}) for the case of families of polarized endomorphisms.
\begin{proof}
    Since $\mu_{\Cscr}\neq 0$, by Proposition \ref{hieghtIntersection}, for a general parameter $t'$
    \begin{align*}
        \lim_{n \to \infty} \frac{d^{-2n} \Cscr_n \cdot {\L'}^2}{d^{-n} \Cscr_n \cdot \L \cdot \pi^*\M} = \frac{\Tilde{h}(\Cscr)}{\deg_{\L'|_{\Cscr_{t'}}}(\Cscr_{t'})\deg{\M}}\eqqcolon C > 0.
    \end{align*}
    Thus there exists a positive integer $N\geq 0$ such that for all $n\geq N$,
    \begin{align*}
        \frac{\Cscr_n\cdot {\L'}^2}{\Cscr\cdot \L'\cdot \pi^*\Mscr} > \frac{Cd^n}{2}.
    \end{align*}
Set $M\coloneqq Cd^N/4,$ we have
\begin{align*}
    \frac{\Cscr_N\cdot {\L'}^2}{2\Cscr_N\cdot \L'\cdot M\pi^*\Mscr}> 1.
\end{align*}
By Siu's numerical criterion for bigness \cite[Theorem 2.2.15]{Lazarsfeld}, $\L' - M\pi^*\M$ is big on $\Cscr_N$. Thus there exists a non-empty Zariski open subset $\Uscr$ of $\Cscr_N$ such that, for all $(t,z)\in \Uscr(\Qbar)$,
    \begin{align*}
        h'(z)-Mh_\M(t)\geq O(1).
    \end{align*}
Hence there exists a non-empty Zariski open subset $\Cscr^0$ of $\Cscr$ such that for all $(t,x)\in \Cscr^0(\Qbar),$
\begin{align}\label{heightinequality3}
    Mh_\M(t)\leq h'\left(\Fscr_N(x)\right) +O(1) \leq h\left(f^N(x)\right)+h\left(f^{-N}(x)\right).
\end{align}
Now, the height inequality \eqref{heightinequality1} implies 
\begin{align}\label{heightinequality4}
    h\left(f^N(x)\right)+h\left(f^{-N}(x)\right) \leq \left(d^N+\frac{1}{d^{N}} \right)\hat{h}_{f_t}(x)+ C_1\left(h_\M(\pi(x))+1\right).
\end{align}
Combining \eqref{heightinequality3} and \eqref{heightinequality4}, there exists a positive constant $C'>0$ such that
\begin{align*}
    \left(\frac{C}{2} - \frac{C_1}{d^N}\right)h_\M(\pi(x))  \leq \left(1+\frac{1}{d^{2N}}\right)\hat{h}_{f_t}(x)+C'. 
\end{align*}
Choose $N$ large enough so that $\frac{C}{2} > \frac{C_1}{d^N}$, we have the desired inequality.
\end{proof}

\section{Equidistribution in families of plane regular polynomial automorphisms}\label{sectionequidistribution}
\subsection{Adelic line bundles on projective varieties}

In this section we give a quick review of the theory of adelic line bundles on projective varieties of Zhang (\cite{zhangJAG1995,ZHANGjams}, see also \cite{ICCMYuan, chambertloir2010heights}). Other references are given in the text.

\subsubsection{Definitions}
Let $\K$ be a number field. Let $M_\K$ be the set of its places, i.e. the set of absolute values on $\K$ whose restriction to $\Q$ are the usual absolute values on $\Q.$ For any $v\in M_\K,$ denote by $\K_v$ the completion of $\K$ w.r.t. $v$ and by $\C_v$ the completion of an algebraic closure of $\K_v$. If $v$ extends $p$, i.e. for any $x\in \Q,$ it holds $|x|_v = |x|_p$, then denote by $n_v\coloneqq[\K_v:\Q_p]$. For any $x\in \K$, we have the product formula $\sum_{v\in M_\K}n_v \log|x|_v=0.$ 

Let $X$ be a smooth projective variety defined over $\K.$ Let $X^\an_v$ be the Berkovich analytification of $X $ over $\C_v.$ Let $L$ be a line bundle on $X,$ and denote by $L^\an$ its analytification. 

A \emph{continuous metric}  $\lVert  \cdot \rVert_v$ on $L^\an_v$ is the data for any open subset $U\subset X^\an_v$ and any section $\sigma$ on $U$ of a continuous function $\lVert \sigma \rVert_v \colon U\to \R_+$ such that
\begin{enumerate}
    \item $\lVert \sigma \rVert_v$ vanishes only at the zeros of $\sigma;$
    \item for any open subset $V\subset U,$ $\lVert \sigma \rVert_v|_V = \lVert \sigma|_V \rVert_v$;
    \item for any analytic function $f$ on $U,$ $\lVert f\sigma \rVert_v = |f|_v \lVert \sigma \rVert_v.$
\end{enumerate}
An \emph{arithmetic model} $(\X,\L,e)$ is a  triple where $\X$ is a proper flat scheme over the ring of integers $\O_\K$ of $\K$ such that $\X_\K\simeq X$ , $e\in \N^*$ and $\L^e_\K \simeq L.$ It induces naturally the \emph{model metric} (see \cite[(1.1)]{zhangJAG1995}) at each non-archimedean place. An \emph{adelic metric}  $\overline{L}\coloneqq(L,\{\lVert \cdot \rVert_v\}_{v\in M_\K})$ on $L$ is a collection of continuous metric at each place $v,$ satisfying some coherent conditions, that is, there exists an arithmetic model $\X$ such that for all but finitely many non-archimedean places $v\in M_\K,$ the metric $\lVert \cdot \rVert_v$ is the model metric induced by $\X.$ 

To any adelic metric $\overline{L},$ we can associate a height function as follows. Let $x\in X(\Qbar)$, denote by $O_v(x)$ its Galois orbit in $X^\an_v$, let $\sigma$ be any rational section of $L$ regular and non-vanishing at $x,$ then
\begin{align*}
    h_{\overline{L}}(x) \coloneqq \frac{-1}{\deg(x)}\sum_{v\in M_\K} n_v\sum_{z\in O_v(x)}\log \lVert \sigma(z) \rVert_v.
\end{align*}

\subsubsection{Positivities of adelic line bundles}
Let $\overline{\L}$ be a Hermitian metric on an arithmetic model $\X$, we say that $\overline{\L}$ is \emph{semipositive} if $\L$ is relatively nef and the curvature form $c_1(\L_\C)$ is semipositive.
An adelic metric $\overline{L}$ is said to be \emph{semipositive} if it is an uniform limit of a sequence of semipositive hermitian metrics, \emph{nef} if moreover $h_{\overline{L}} \geq 0$, and integrable if it is the difference of two nef adelic line bundles. 
\begin{ex} \normalfont \label{standradmetric}
    Let $X=\P^n_\K$ with coordinates $z_i$ and $L=\O(1)$. Let $\sigma \in H^0(X,L)$ be a section, it can be represented by a linear form $\sum_{i=1}^n a_iz_i,$ where $a_i \in \K$. Then at each place $v\in M_\K,$ we can define a metric by setting
    \begin{align*}
        \lVert \sigma(z_0,\cdots,z_n)\rVert_v = \frac{|\sum_{i=1}^n a_iz_i|_v}{\max\{|z_0|_v,\cdots,|z_n|_v \}}.
    \end{align*}
    The adelic metric $\overline{L} = (L, \{\lVert \cdot \rVert\}_v)$ is semipositive, and we call it the \emph{standard metric} on the line bundle $L$. In fact, it's associated height function $h_{\overline{L}}$ is the \emph{standard Weil height function}, i.e. for any $z\in X(K)$,
    \begin{align*}
        h_{\overline{L}}(z) = \frac{1}{[\K:\Q]}\sum_{v\in M_\K} n_v \log \max\{|z_0|_v,\cdots,|z_n|_v\}.
    \end{align*}
\end{ex}
Denote by $\widehat{\mathrm{Pic}}(X)_{\mathrm{int}}$ the group of integrable adelic line bundles. For any positive integer $d\geq 0,$ denote by $Z_d(X)$ the group of Chow cycles of dimension $d$. Then
we have an intersection pairing $\widehat{\mathrm{Pic}}(X)_{\mathrm{int}}^{d+1} \times Z_d(X) \to \R$. We write the product as $\overline{L}_0 \cdots \overline{L}_d\cdot Y$. If $Y=X,$ then we can simply denote it by $\overline{L}_0 \cdots \overline{L}_d$

Following (\cite{ZHANGjams,zhangJAG1995}), we define the absolute minimum of $\overline{L}$ by $e_{\mathrm{abs}}(\overline{L})\coloneqq\inf_{x\in X(\Qbar)}h_{\overline{L}}(x)$,
and the essential minimum of $\overline{L}$ by $e_{\mathrm{ess}}(\overline{L})\coloneqq\sup_{Y\not\subset X}\inf_{x\in X\setminus Y(\Qbar)}h_{\overline{L}}(x)$,
where the supremum is taken over all closed proper subvarieties $Y$. The following fundamental inequalities are due to S.-W. Zhang.
\begin{thm}[{\cite[Theorem (1.10)]{zhangJAG1995}}]\label{zhanginequality}
    If $\overline{L}$ is semipositive and $L$ is big and nef, then
    \begin{align*}
        e_{\mathrm{ess}}(\overline{L})\geq \frac{\overline{L}^{\mathrm{dim}(X)+1}}{(\mathrm{dim}(X)+1)D^{\mathrm{dim}(X)}} \geq \frac{1}{\mathrm{dim}(X)+1}(e_{\mathrm{ess}}(\overline{L}) + \mathrm{dim}(X)e_{\mathrm{abs}}(\overline{L})).
    \end{align*}
    In particular, if $e_{\mathrm{abs}}(\overline{L})\geq 0$, then $\overline{L}^{\mathrm{dim}(X)+1} \geq 0.$
\end{thm}

Set $\hat{H}^0(X,\overline{L}) \coloneqq \{\sigma\in H^0(X,L) \ |\ \lVert \sigma \rVert_v\leq 1 ,\ \forall v\in \M_\K\}$ and $\hat{h}^0(X,\overline{L}) \coloneqq \log \mathrm{Card}(\hat{H}^0(X,\overline{L}))$. The \emph{arithmetic volume} of $\overline{L}$ is 
\begin{align*}
    \widehat{\vol}(\overline{L}) = \limsup_{n\to +\infty}\frac{\hat{h}^0(X,\overline{L})}{n^{\mathrm{dim}(X)+1}/(\mathrm{dim}(X)+1)!}.
\end{align*}
We say that $\overline{L}$ is \emph{big} if $ \widehat{\vol}(\overline{L})>0$ and \emph{pseudo-effective} if for any big adelically metrized line bundle $\overline{E}$, $\overline{L}+\overline{E}$ is big. 

The following theorem is proved by Balla\"{y} (\cite{ballay2021}) when $L$ is big and $\overline{L}$ is semipositive (in fact he proved the equivalence in this case), the general case is due to Qu and Yin (\cite{qu2023arithmetic}).
\begin{thm}\label{essentialmin}
    We have the implication $ e_{\mathrm{ess}}(\overline{L}) \geq 0 \implies \overline{L}$ is pseudo-effective. 
\end{thm}
The following positivity theorem is due to A. Moriwaki.
\begin{prop}[{\cite[Proposition 4.5.4]{MoriwakiRdivisor}}]\label{pesudoeffective}
    If $\overline{L}_1,\dots,\overline{L}_{\mathrm{dim}(X)}$ are nef and $\overline{L}_{\mathrm{dim}(X)+1}$ is pseudo-effective, then $\overline{L}_1\cdots\overline{L}_{\mathrm{dim}(X)+1} \geq 0.$
\end{prop}
To finish these two preliminary subsections, we remark that the theory works also for $\Q$-line bundles and all the line bundles we consider in this text is rational unless explicitly stated to the contrary.

\subsection{Equidistribution on quasi-projective varieties}
Let $V$ be a smooth quasi-projective variety over a number field $\K$ of dimension $k$. Fix an archimedean place $v\in M_\K$. We define an $v$-\emph{adic equidistribution model} $\mathrm{EQ}_v\big(V,(V_n)_{n\geq 0},(\Psi_n)_{n\geq 1},(\overline{L}_n)_{n\geq 1}\big)$
(or $\mathrm{EQ}(V)$ if it is clear from the context) of $V$ by the following data:
\begin{itemize}
    \item $V_n$ is a projective variety over $\K$.
    \item There exists an open immersion $\imath_n \colon V \to V_n$ such that $\Psi_n \colon V_n \to V_0$ is a birational morphism which is an isomorphism on $V.$
    \item $L_n$ is a big and nef line bundle on $V_n,$ endowed with an semipositive adelic metric $\overline{L}_n.$
\end{itemize}
We say that $\mathrm{EQ}(V)$ is 
\begin{itemize}
    \item \emph{non-degenerate} if we have $\lim_{n\to +\infty}\vol(L_n) >0$, and the sequence of probability measures $\vol(L_n)^{-1}(\Psi_n)_* c_1(\overline{L}_n^k)$ converges weakly to some probability measure $\mu_v$ on $V^{\an}_v$.
    \item \emph{bounded} if either $k=1$, or $k>2$ and for any ample line bundle $E$ on $V_0$ with a nef adelic metrization $\overline{E}$, there exists a positive constant $C>0$ such that 
\begin{align*}
    \big(\Psi_n^*(\overline{E}) \big)^j \cdot (\overline{L}_n)^{k+1-j} \leq C,
\end{align*}
for any $2\leq j \leq k$ and any $n\geq 1.$ 
\end{itemize}

Given a sequence $x_m \in V(\Qbar)$, it is  called \emph{generic} if for any subvariety $W$ of $V$, there exists some integer $N\geq 1$ such that, for all $m\geq N$, $x_m \notin W(\Qbar).$ For any generic sequence $x_m$, we say that it is \emph{$\mathrm{EQ}(V)$-small} if 
\begin{align*}
    \lim_{n\to +\infty}\Big( \limsup_{m\to +\infty}\big(h_{\overline{L}_n}(\Psi_n^{-1}(x_m))\big) - h_{\overline{L}_n}(V_n)\Big) = 0.
\end{align*}

The following is a reformulation by Gauthier \cite[Theorem 2]{gauthier2023good} of the equidistribution theorem of Yuan and Zhang \cite[Theorem 5.4.3]{YZadelic}:
\begin{thm}\label{equidistributionthomas}
    Let $V$ be a smooth quasi-projective variety over a number field $\K.$  Let $v\in M_\K$ be an archimedean place. Let $\mathrm{EQ}(V)$ be a non-degenerate and bounded equidistribution model of $V.$ Let $(x_m)_{m\geq 1}$ be a generic and $\mathrm{EQ}(V)$-small sequence in $V(\Qbar)$. 
    
    Then the sequence of probability measures $\mu_{x_m}$ that are uniformly supported on the Galois orbit $\gal(\overline{\Q}/\K)x_m$ of $x_m$ converges weakly to $\mu_v,$ i.e. for any continuous and compactly supported function $\varphi$ on $V^\an_v$, we have
    \begin{align*}
        \lim_{m \to +\infty} \frac{1}{\mathrm{Card}(\gal(\overline{\Q}/\K)x_m)}\sum_{y\in\gal(\Qbar/\K)x_m} \varphi(y) = \int_{V^\an_v} \varphi \mu_v.
    \end{align*}
\end{thm}

\subsection{Equidistribution for non-isotrivial marked points}
Let $f_t$ be an algebraic family of plane regular polynomial automorphisms of degree $d\geq 2$ parameterized by a smooth complex quasi-projective curve $\Lambda$. Let $\sigma\colon\Lambda \to \A^2$ be a marked point. We say that the family $f_t$ is \emph{isotrivial} if there exist a finite morphism $\alpha\colon \Lambda' \to \Lambda$ and a family of affine polynomial automorphisms $\phi_t$ parameterized by $\Lambda'$, then the conjugate family $\phi_t\circ f_{\alpha(t)}\circ\phi_t^{-1}$ is constant, and the pair $(f_t,\alpha)$ is \emph{isotrivial} if moreover $\sigma\circ \alpha$ is constant.

\begin{thm}\label{boundedheight}
    Let $f_t$ be an algebraic family of polynomial automorphisms of H\'enon type of degree $d\geq 2$ parameterized by a smooth quasi-projective curve $\Lambda$, defined over a number field $\K.$ Let $\sigma$ be a marked point. Then 
    \begin{align*}
        \{t\in \Lambda(\Qbar) \mid \sigma(t)\ \text{is periodic for}\ f_t \}
    \end{align*}
    is a set of bounded height.
\end{thm}
\begin{proof}
Let $\Lbar$ be the pull back of the canonical line bundle $\overline{\O_{\P^2}(1)}$ with the standard adelic metrization (see Example \ref{standradmetric}) by the projection $\pi_{\P^2}\colon \B\times \P^2 \to \P^2$. For any positive integer $n\geq 1$, define an adelic metrization by $\Dbar^+_n \coloneqq 1/d^n {(f^n\circ \sigma)}^{*} \Lbar$. We may assume that there exist infinitely many periodic parameters $t_m$. By the height inequality \eqref{heightinequality1}, there is an Zariski open subset $U$ of $\Lambda$ such that, for all $t\in U(\Qbar)$, we have $h_{\Dbar^+_n}(t_m) \leq \frac{C_1}{d^n}(h_M(t_m) +1)$, and
\begin{align}\label{UpperBound}
    \limsup_{m\to +\infty}h_{\Dbar^+_n}(t_m) \leq  \limsup_{m\to +\infty} \frac{C_1}{d^n}(h_M(t_m) +1). 
\end{align}
Let $N$ be a positive integer such that $N\D_n - M$ is ample for all $n$, there exists a constant $c=c(n)$ such that
\begin{align}\label{LowerBound}
    \frac{h_M(t_m)}{N} + \frac{c}{N} \leq h_{\Dbar^+_n}(t_m).
\end{align}
If $h_M(t_m)$ is not bounded, the inequalities ($\ref{UpperBound}$) and ($\ref{LowerBound}$) would imply $N\leq C_1/d^n$ for all $n$, leading to a contradiction.
\end{proof}
\begin{rem}\label{CompareIngram}
    As mentioned in Introduction,  this result generalizes Ingram's work \cite[Theorem 1.3]{Ingram2011CanonicalHF}. In his case, Ingram established \cite[Theorem 1.1]{Ingram2011CanonicalHF} that $\hat{h}_{f_t}$ is in fact a Weil height associated with an ample line bundle. Since our primary interest lies in equidistribution results, we focused on proving height inequalities, as presented in Lemma~\ref{heightinequalitylemma1}. This approach is, in some respects, both more general and less general than \cite[Theorem 1.1]{Ingram2011CanonicalHF}.
\end{rem}

\begin{thm}\label{equidistributionpoints}
    Let $f_t$ be an algebraic family of regular plane polynomial automorphisms of degree $d\geq 2$ parameterized by a smooth quasi-projective curve $\Lambda$ defined over a number field $\K$ and $\sigma\colon\Lambda \to\Lambda\times \A^2$ a marked point defined over $\K.$ Fix an archimedean place $v$ of $\K$, so that we have an embedding $\K\hookrightarrow\C.$ Suppose the pair $(f_t,\sigma)$ is non-isotrivial and non-periodic.

    Then the following is true. If there exists a non-eventually constant sequence of parameters $t_m\in \Lambda(\Qbar)$ such that $\lim_{m\to +\infty}\hat{h}_{f_{t_m}}(\sigma(t_m))=0$, 
    then the sequence of probability measures $\mu_{t_m}$ --- that are uniformly supported on the Galois orbits $\gal(\overline{\Q}/\K)t_m$ of $t_m$ --- converges weakly to both $\mu^+_{f,\sigma}$ and $\mu^-_{f,\sigma}$ on $\Lambda(\C)$, up to some possibly different positive multiplicative constants.
\end{thm}
\begin{proof}
    Let $\Lbar$ be the pull back of the canonical line bundle $\overline{\O_{\P^2}(1)}$ with the standard adelic metrization (see Example \ref{standradmetric}) by the projection $\pi_{\P^2}\colon \B\times \P^2 \to \P^2$, and $\omega_\L$ its curvature form.
    For any positive integer $n\geq 1$, define an adelic metrization by $\Dbar^+_n \coloneqq 1/d^n {(f^n\circ \sigma)}^{*} \Lbar$ and denote by $\omega^+_n$ its curvature form. Our equidistribution model will be 
    \begin{align*}
        \mathrm{EQ}(\Lambda)\coloneqq\mathrm{EQ}_v(\Lambda,\B,\mathrm{id}_\B,\Dbar_n).
    \end{align*}

    Let us first show that $\mathrm{EQ}(\Lambda)$ is non-degenerate. If $f_t$ is non-isotrivial, then since $\sigma$ is non-periodic, $\lim_{n\to +\infty}\vol(\D_n^+)\neq 0$ by \cite{GVhenon}. By the construction of the Green function $G_f^+$, we have $\lim_{n \to +\infty}\omega^+_n = \mu^+_{f,\sigma}$, and again this is non-zero by \cite{GVhenon}. If $f_t$ is isotrivial, we can assume $f_t=g$ for some fixed $g$ for all $t\in \Lambda(\Qbar)$. Since ($f_t,\sigma$) is not-isotrivial, we have $\lim_{n\to +\infty} \vol(\D_n^+)=\deg(\sigma)\neq 0$. Let $\omega_{\mathrm{FS}}$ be the curvature form of $\overline{\O_{\P^2}(1)},$ then \[\lim_{n\to +\infty}\omega^+_n =\sigma^* \pi_{\P^2}^* \left(\lim_{n\to +\infty}{g^n}^* \omega_{\mathrm{FS}}/d^n\right) = \mu^+_{f,\sigma}\neq 0.\]

    To apply Theroem~\ref{equidistributionthomas}, it remains   to show that the sequence $t_m$ is $\mathrm{EQ}(\Lambda)$-small. Since $\Dbar_n$ is nef, Zhang's inequality (Theorem \ref{zhanginequality}) implies that $h_{\Dbar^+_m}(\B) \geq 0$. Hence, by Theorem~\ref{boundedheight} and inequality~\eqref{UpperBound}, there exists a positive constant $c'>0$ such that
\begin{align*}
    \limsup_{m \to +\infty}h_{\Dbar^+_n}(t_m) - h_{\Dbar^+_n}(\B) \leq \frac{c'}{d^n},
\end{align*}
which proves the smallness of $t_m$ and, by Theorem~\ref{equidistributionthomas}, the convergence of the measure $\mu_{t_m}$ to some positive multiple of $\mu_{f,\sigma}^+$.

We can now repeat the argument above for $\Dbar^-_m\coloneqq 1/d^m {(f^{-m}\circ \sigma)}^{*} \Lbar$, obtaining the convergence of $\mu_{t_m}$ to a positive multiple of $\mu_{f,\sigma}^-$.
\end{proof}

\subsection{Equidistribution for non-degenerate families of curves}
Recall the good open subset $\Cscr^0$ of $\Cscr$ given by Theorem~\ref{height inequality curve}. Recall also the notations in subsection~\ref{notations}
\begin{thm}\label{equidistribution curves}
    Let $\Cscr\subset \Lambda \times \A^2$ be a non-degenerate family of curves defined over a number field $\K.$ Fix an archimedean place $v\in M_\K$, so that we have an inclusion of $\K\hookrightarrow\C.$ Suppose we have a generic sequence $x_m=(t_m,z_m)\in\Cscr^{0,[2]}(\Qbar)\subset \Lambda\times\A^4(\Qbar)$ such that $\lim_{m\to +\infty}\hat{h}^{[2]}_{f_{t_m}}(x_m)=0.$
    
    Then the sequence of probability measures $\mu_{x_m}$, which are uniformly supported on the Galois orbit $\gal(\overline{\Q}/\K)x_m$ of $x_m$, converges weakly on $\Cscr^{0,[2]}(\C)$ to some positive multiple of
    \begin{align*}
        \mu_{f^{[2]},\Cscr^{0,[2]}}\coloneqq\left( \ddc G_{f}^{[2]} \right)^3 \wedge [\Cscr^{0,[2]}]
    \end{align*}
\end{thm}

\begin{proof}
    Let us first construct an equidistribution model of the smooth irreducible variety $\Cscr^{0,[2]}$. Let $\Tilde{\Cscr}_0$ be the Zariski closure of $\Cscr^{0,[2]}$ in $\B\times\P^4$. Let $\Tilde{\Cscr}_n$ be the normalization of the graph of the birational map $\Fscr_n^{[2]}\colon \Tilde{\Cscr}_0 \dashrightarrow \Cscr_n^{[2]}$. Let $\Tilde{\Psi}_n\colon\Tilde{\Cscr}_n \to \Tilde{\Cscr}_0$ and $\Tilde{\Gscr}_n\colon\Tilde{\Cscr}_n \to \Cscr_n^{[2]}$ be the two projections and $\Tilde{\pi}_n\coloneqq{\pi'} \circ \Gscr_n \colon \Tilde{\Psi}_n \to \B$ the projection to the base $\B.$ 
    \[
    \begin{tikzcd}
& \Tilde{\Cscr}_n \arrow[ld, "\Tilde{\Psi}_n"'] \arrow[rd, "\Tilde{\mathscr{G}}_n"] \arrow[dd, "\Tilde{\pi}_n" description, bend right] &    \\
\Tilde{\mathscr{C}}_0\subset\mathfrak{B}\times\mathbb{P}^4 \arrow[rr, "{\quad\quad\ \ \ \mathscr{F}_n^{[2]}}", dashed] \arrow[rd, "\pi'"'] &                                                                                                                                   & {\mathscr{C}^{[2]}_n\subset\mathfrak{B}\times\mathbb{P}^4}\times\mathbb{P}^4 \arrow[ld, "{\pi'}^{[2]}"] \\
                        &\mathfrak{B}      &            
\end{tikzcd}
    \]
    Let $\Mbar$ be a semipositive adelic metric on an ample divisor $\M$ such that $h_{\Mbar}(t)\geq 0$ for all $t\in \B(\Qbar)$.
    Let $\Lbar'$ be the pull back of the canonical line bundle $\overline{\O_{\P^4}(1)}$ with the standard adelic metrization (Example \ref{standradmetric}) by the projection $\B\times \P^4 \to \P^4$.
    Let $\Lbar'_n\coloneqq \frac{1}{d^n}\left(\Tilde{\Gscr}_n^{*}\Lbar'^{[2]}+\Tilde{\pi}_n^{*}\Mbar\right)$. We thus have a model
    $$\mathrm{EQ}(\Cscr^{0,[2]})=\mathrm{EQ}_v(\Cscr^{0,[2]}, \Tilde{\Cscr}_n, \Tilde{\Psi}_n, \Lbar'_n).$$

    We first show that the model $\mathrm{EQ}(\Cscr^{0,[2]})$ is non-degenerate. The volume of $\L'_n$ is 
    $$
    \vol(\L'_n) = \frac{1}{d^{3n}}\left(\Tilde{\Gscr}_n^{*}\Lbar'^{[2]}+\Tilde{\pi}_n^{*}\Mbar\right)^3\cdot \Cscr_n^{[2]} = \frac{1}{d^{3n}}\left({\L'}^{[2]}\right)^3\cdot \Cscr_n^{[2]} +\frac{1}{d^{3n}}O(1),
    $$
    where $O(1)$ is independent of $n.$ Since 
    \begin{align*}
        \left( \ddc G_{f^{[2]}} \right)^3 \wedge [\Cscr^{0,[2]}] = p^*\mu_{f,\Cscr^0}\wedge q^*\left(\ddc G \wedge [\Cscr^0]\right)
    \end{align*}
    and by assumption $\mu_{f,\Cscr}\neq 0,$ we infer that the measure $ \mu_{f^{[2]},\Cscr^{0,[2]}}$ is non zero either.Therefore the limit of volumes $\vol(\L'_n)$ is non zero. In fact, we have 
    \begin{align*}
        \lim_{n\to+\infty} \vol(\L'_n) =\lim_{n\to+\infty} \frac{1}{d^{3n}}\left({\L'}^{[2]}\right)^3\cdot \Cscr_n^{[2]} \geq \int_{\Lambda\times\A^2\times\A^2(\C)}  \left( \ddc G_{f^{[2]}} \right)^3 \wedge [\Cscr^{0,[2]}] >0.
    \end{align*}
   
    Let $\omega$ be the curvature form $c_1(\Lbar')$ on $\Lambda\times \A^4(\C)$ and $\omega_b$ the curvature form $c_1(\Mbar)$ on $\Lambda(\C)$. By the construction of the Green function $G_f$, we have
    \begin{align*}
        \lim_{n\to\infty}\frac{1}{d^n}\bigg({\Fscr_n^{[2]}}^{*}(\omega^{[2]})+{\pi'}^{*}(\omega_b) \bigg) = dd^cG_f^{[2]}.
    \end{align*}
    Thus the measure $\mu_v$ of the equidistribution model is $\mu_{f^{[2]},\Cscr^{0,[2]}}$.

    We next show that the sequence $x_m$ is $\mathrm{EQ}(\Cscr^{0,[2]})$-small. 
    By the inequality~\eqref{heightinequality4}, for all positive integers $n\geq 0$ and for all $x=(t,z)\in \Cscr^{0,[2]}(\Qbar)$, we have
    \begin{align}\label{heightinequality5}
        \frac{1}{d^n}\bigg(h^{[2]}\big((f_t^{[2]})^{n}(z)\big) + h^{[2]}\big((f_t^{[2]})^{-n}(z)\big)\bigg) \leq \left(1+\frac{1}{d^{2n}}\right)\hat{h}_{f_t^{[2]}}(z) + \frac{2C_1}{d^n}(h_{\Mbar}(t)+1).
    \end{align}
    Since by construction $h_{\Tilde{\Cscr}_n,\Lbar'_n}(x)\geq 0$ for any $x\in \Tilde{\Cscr}_n(\Qbar)$, Theorem~\ref{zhanginequality} implies $h_{\Tilde{\Cscr}_n,\Lbar'_n}(\Tilde{\Cscr}_n)\geq 0.$ 
    Thus for all $x=(t,z)\in \Cscr^{0,[2]}(\Qbar)$, we have
    \begin{align*}
        \varepsilon_n(x)&\coloneqq h_{\Tilde{\Cscr}_n, \Lbar'_n}(\Tilde{\Psi}_n^{-1}(x)) -h_{\Tilde{\Cscr}_n, \Lbar'_n}(\Cscr_0) \leq  h_{\Tilde{\Cscr}_n, \Lbar'_n}\big(\Tilde{\Psi}_n^{-1}(x)\big) = \frac{1}{d^n}\Big((h')^{[2]}\big(\Fscr^{[2]}_{n,t}(z)\big)+h_{\Mbar}(t)\Big)\\
        &\leq   \frac{1}{d^n}\Big(h^{[2]}\big((f_t^{[2]})^{n}(z)\big) + h^{[2]}\big((f_t^{[2]})^{-n}(z)\big)+h_{\Mbar}(t)\Big)
    \end{align*}
    By \eqref{heightinequality5} and Lemma \ref{height inequality curve}, there exist constants $c_1,c_2>0$ such that $\varepsilon_n(x)\leq c_1\hat{h}_{f_t^{[2]}}(x) +\frac{c_2}{d^n}.$
Since $\lim_{m\to \infty}\hat{h}^{[2]}_{f_{t_m}}(x_m) =0$ by our assumption, we have $\limsup_{m\to +\infty} \varepsilon_n(x_m) \leq \frac{c_2}{d^n}$, which tends to zero when $n\to +\infty.$

Finally we need to verify that $\mathrm{EQ}(\Cscr^{0,[2]})$ is bounded.
Let $\E$ be any ample line bundle on $\Tilde{\Cscr}_0$ with a nef adelic metrization $\overline{\E}$. By the height inequality~\eqref{inequality2}, there exists a nef adelic metrization $\Mbar'$ of $\M$ and a positive integer $C_M>0$ such that $e_{\mathrm{ess}}(C_M \Tilde{\pi}_n^*\Mbar' +2\Tilde{\Psi}_n^*\Lbar^{[2]}-\Lbar'_n)\geq 0$. By Theorem \ref{essentialmin}, this metrized line bundle is pseudo-effective, and by Proposition \ref{pesudoeffective},
\begin{align*}
    \left(\Tilde{\Psi}_n^*\overline{\E}\right)^2\cdot\Lbar'_n\leq \left(\Tilde{\Psi}_n^*\overline{\E}\right)^2 \cdot (C_M\Tilde{\pi}_n^*\Mbar' +2\Tilde{\Psi}_n^*\Lbar^{[2]}) =\overline{\E}^2 \cdot (C_M{\pi^{[2]}}^* \Mbar'+2\Lbar^{[2]}).
\end{align*}
All the conditions of Theorem \ref{equidistributionthomas} are satisfied, leading to equidistribution.
\end{proof}

\section{Proofs of main theorems}\label{sectionproof}
\subsection{Proof of Theorem \ref{MainMarkedPoint}}
We may assume that the pair $(f_t,\sigma)$ is non-isotrivial; otherwise we reduces to a single regular plane polynomial automorphism and a non-periodic point, implying the absence of any periodic parameters. Suppose there exist infinitely many periodic parameters $t$ such that $\sigma(t)$ is periodic for $f_t$. In particular, this means $\hat{h}_{f_t}(\sigma(t))=0$. Then, by the equidistribution theorem \ref{equidistribution curves}, the two non-vanishing Green measures $\sigma_{f,\sigma}^\pm$ are proportional. Therefore, we can apply the results of Sect.~\ref{sectionmarkedpoints}. The assumption that $\Jac(f_t)$ lies on the unit circle implies that $\sigma(t)$ can not be neutral periodic and moreover, can not be saddle either, as established by Proposition~\ref{saddle parameters}. Proposition~\ref{propsemi} further shows that $\sigma(t)$ is neither semi-repelling nor semi-attracting. Finally $\sigma(t)$ can not be repelling or attracting by Proposition~\ref{proprepelling}. Thus, we reach a contradiction, as none of the possible situations are viable.

\subsection{Proof of Theorem \ref{MainMarkedPoint2}}
We can assume the pair $(f_t,\sigma)$ is non-isotrivial. Suppose there exist infinitely many periodic parameters $t\in \Lambda$. By the proof of Theorem \ref{MainMarkedPoint}, no parameter can be (semi-)repelling or (semi-)attracting. Now, suppose there are infinitely many saddle parameters. Then by Proposition \ref{saddle parameters}, we know that the Jacobians $\Jac(f_t)$ must all be roots of unity. Denote by $C\subset \Lambda$ the real analytic curve formed by the parameters $t$ for which $f_t$ has Jacobian lying on the unit circle. Applying the equidistribution theorem \ref{equidistribution curves} to these parameters, we obtain $\supp(\mu_{f,\sigma}^+) \subset C$. Let $t_0$ be a parameter such that $\sigma(t_0)$ is a saddle fixed point and $C$ is smooth at $t_0.$ We now work in the local coordinate as described in the step 0 of the proof of Proposition \ref{saddle parameters}. Consider the support
\begin{align*}
    Z_n\coloneqq\{t\in \Dbb(\varepsilon) \ |\  \frac{t}{\lambda_u^n } \in \supp(\mu_{f,\sigma}^+)\}
\end{align*}
of $dd^c G^+_{r^u_n(t)}(\sigma(r^u_n(t)))$. Denote by $Z^+$ the support of $dd^c(G^+_{0}\circ \rho_0^u)$. Let $\phi_q\colon t\mapsto t^q$, then $\phi_q^{-1} Z^+$ is contained in the limit points of $Z_n$, which is the set of points $t\in \Dbb(\varepsilon)$ such that there exists a sequence $t_n$ of points in $Z_n$ coverging to $t$.
Write $\lambda_u=|\lambda_u|e^{i\pi\theta}$ and
let $\xi^{-1}$ be the limit of a subsequence $e^{i\pi\theta n_j}$ of $e^{i\pi\theta n}$. Let $L$ be the tangent line of $C$ at the origin (in $\Dbb(\varepsilon)$). The limit points of $Z_{n_j}$ is contained $\xi L$, hence $Z^+ \subset \phi_q(\xi L)$. This implies that $e^{i\pi\theta n}$ has only two limit points, $\pm 1$. Therefore, $e^{i\pi\theta}=\pm 1$, meaning $\lambda_u$ is real and $Z^+ \subset \phi_q(L)$. 
We will need the following lemma.
\begin{lem}[{\cite[Proposition 2.2]{CantatDujardin}}]\label{RigidityDujardinCantat}
    Let $f$ be a regular polynomial automorphism of complex affine plane. Denote by $G^+_f$ the forward Green function of $f.$ Let $p$ be a saddle periodic point. Let $W^u_{\mathrm{loc}}(p)$ be its unstable local manifold and $\rho_p\colon \Dbb \to W^u_{\mathrm{loc}}(p)$ a local parametrization (thus $\rho_p(0)=p$), where $\Dbb\subset \C$ is the unit disk.  If $\supp(dd^c G_f^+\circ \rho_p)$ is contained in a line passing though the origin, then for any other saddle periodic point $q$, $\supp(dd^c G_f^+\circ \rho_q)$ is also contained in a line passing through the origin.
\end{lem}

In the local coordinate provided by a parametrization of a local unstable manifold, the action of $f$ on the Julia set corresponds to multiplication by the unstable multiplier. If the Julia set is contained in a line, then the unstable multiplier has to be real. Let $\sigma_0$ denote the local analytic continuation of the saddle  point $\sigma(t_0)$ over an analytic open subset $U\subset \Lambda$. Take any saddle parameter $t$ close to $t_0$ and repeat the above argument. This implies that the Julia set in the local unstable manifold of $\sigma(t)$ is contained in a line through the origin. By Lemma \ref{RigidityDujardinCantat}, the same property holds at the saddle fixed point $\sigma_0(t)$. By symmetry, the Jacobian of $f_t$ must therefore be real. Since $\Jac(f_t)$ is a root of unity, it must be $\pm 1.$ However, a holomorphic function of one variable has discrete preimages, leading to a contradiction.

\subsection{Proof of Theorem \ref{maincurve}}
Since $\Cscr$ is non-degenerate, Proposition \ref{positivityI_f} implies the following. Up to replacing $f$ by an iterate, there exists an analytic open subset $U\subset \Lambda\times \A^4(\C)$ such that, $G_{f^{[2]}}|_U >0$ and $\mu_{f,\Cscr^{0,[2]}}|_U>0.$ Let $\chi \colon U \to \R_+$ be a test function on $U$ such that $0\leq \chi \leq 1$. By construction of the Green functions, for any $(t,x,y)\in \Lambda\times(\A^2)^2(\Qbar)$, we have 
\begin{align}\label{inequality1.3.1}
\begin{split}
    \hat{h}_{f^{[2]}_t}(x,y)=\hat{h}_{f_t}(x)+\hat{h}_{f_t}(y) &\geq G^+_{f_t}(x)+G^-_{f_t}(x)+G^+_{f_t}(y)+G^-_{f_t}(y)\\
    &\geq \frac{1}{2}(G_{f_t}(x) +G_{f_t}(y)) =\frac{1}{2}G_{f^{[2]}_t}(x,y).
\end{split}
\end{align}
Suppose by contradiction that for every $n\in \Z_{\geq 1}$, there exists a parameter $\lambda_n\in \Lambda(\Qbar)$ such that the set $W_n \coloneqq \{(\lambda_n,w)\in \Cscr_{t_n}(\Qbar)\ |\ \hat{h}_{f_{t_n}}(w) \leq 1/m\}$ contains at least $n$ points. Order the set of points of the form $(\lambda_m,w_1,w_2)$ with $w_1,w_2\in W_m$ for some $m\geq 1$, and denote this sequence by $x_n=(t_n,z_n)\in \Lambda\times \A^4(\Qbar)$. We claim that $x_n$ is a generic sequence in $\Cscr^{[2]}$, and thus in $\Cscr^{0,[2]}$. In fact, otherwise the Zariski closure $\Zscr$ of $x_n$ in $\Cscr^{[2]}$ would have relative dimension 1 with respect to the base $\Lambda$. The projection $\Zscr \to \Cscr$ would then a finite morphism with degree $d_1$. Since $x_n$ is symmetric (i.e., if we write $x_n=(t_n,w_1,w_2)$, then there exists $n'$ such that $x_{n'}=(t_n,w_2,w_1)$), there are at most $d_1^2$ points in $W_n$ for all but finitely many $n$, which leads to a contradiction. 

Applying the equidistribution Theorem \ref{equidistribution curves}
to the test function $\chi G_{f^{[2]}},$ for all $n$ large enough, we have
\begin{align}\label{inequality1.3.2}
\frac{1}{\mathrm{Card}(\O(x_n))}\sum_{x'\in \O(x_n)}\chi(x')G_{f^{[2]}}(x') \geq \frac{1}{2}\int_{\Lambda\times(\A^2)^2(\C)} \chi G_f \,\mu_{f^{[2]},\Cscr^{0,[2]}}>0,
\end{align}
where $\O(x_n)$ is the Galois orbit $\gal(\Qbar/K) x_n$ of $x_n$. Since height functions are invariant by the action of the Galois group, combining \eqref{inequality1.3.1} and \eqref{inequality1.3.2}, we have
\begin{align*}
    \hat{h}_{f^{[2]}_t}(x_n) &\geq \sup_{x'\in \O(x_n)}\frac{1}{2}G_{f^{[2]}}(x') \geq \frac{1}{4}\int_{\Lambda\times(\A^2)^2(\C)} \chi G_f\, \mu_{f^{[2]},\Cscr^{0,[2]}} >0,
\end{align*}
which is a contradiction.

\subsection{Proof of Theorem~\ref{MainBoundedHeightSurface}}
The assertion that $\{t\in\Lambda(\Qbar) \mid \exists z\in \Cscr_t(\Qbar), z\ \text{is periodic for}\ f_t \}$ has bounded height follows directly from Lemma~\ref{height inequality curve}. It remains to show that there exists a positive constant $C>0$ such that, for any $f$-periodic point $z=(t,x)$, we have $h(x)\leq C$. By the height inequality~\eqref{heightinequality1} of Lemma~\ref{heightinequalitylemma1}, this is true on a Zariski open subset $\Cscr^0$ of $\Cscr$. The subset $Z\coloneqq \Cscr\setminus \Cscr^0$ is of dimension at most 1. Theorem~\ref{MainBoundedHeight}, combined with height inequality~\eqref{heightinequality1}, imply that the periodic points contained in $Z$ is of bounded height as well, completing the proof.

\subsection{Proof of Theorem \ref{maincurve2}}
By Proposition \ref{DegenerationFilledJulia}, there exists $t_0 \gg 1$ such that for all $|t|\geq t_0$, we have $r_t < r$ and $\Cscr_t \cap K_t = \emptyset.$ Thus, condition (1) of Proposition \ref{comparaisonPrinciple} is satisfied. Now, suppose for the sake of contradiction that for any $n\in \Z_{\geq 1}$, there exists $t_n\in \Lambda(\Qbar)$ such that the set $\{z\in \Cscr_{t_n}\ |\ \hat{h}_{f_{t_n}}(z) \leq 1/m\}$ contains at least $n$ points. this assumption, combined with the already verified condition (1), implies the condition (2) of Proposition \ref{comparaisonPrinciple}. Thus $\Cscr$ is non-degenerate and we can apply Theorem \ref{maincurve}, leading to a contradiction.

\bibliographystyle{plain}
\bibliography{bib}

\begin{thebibliography}{10}

\bibitem{ballay2021}
Fran\c{c}ois Balla\"{y}.
\newblock Successive minima and asymptotic slopes in {A}rakelov geometry.
\newblock {\em Compos. Math.}, 157(6):1302--1339, 2021.

\bibitem{BedfordComparisonPrincipal}
Eric Bedford.
\newblock The operator (ddc)n on complex spaces.
\newblock In Pierre Lelong and Henri Skoda, editors, {\em S{\'e}minaire Pierre Lelong-Henri Skoda (Analyse) Ann{\'e}es 1980/81}, pages 294--323, Berlin, Heidelberg, 1982. Springer Berlin Heidelberg.

\bibitem{BLS93}
Eric Bedford, Mikhail Lyubich, and John Smillie.
\newblock Polynomial diffeomorphisms of {${\bf C}^2$}. {IV}. {T}he measure of maximal entropy and laminar currents.
\newblock {\em Invent. Math.}, 112(1):77--125, 1993.

\bibitem{BS1}
Eric Bedford and John Smillie.
\newblock Polynomial diffeomorphisms of c2: currents, equilibrium measure and hyperbolicity.
\newblock {\em Inventiones mathematicae}, 103(1):69--100, 1991.

\bibitem{BS3}
Eric Bedford and John Smillie.
\newblock Polynomial diffeomorphisms of {$\bold C^2$}. {III}. {E}rgodicity, exponents and entropy of the equilibrium measure.
\newblock {\em Math. Ann.}, 294(3):395--420, 1992.

\bibitem{Benedetto2007}
Robert~L. Benedetto.
\newblock Preperiodic points of polynomials over global fields.
\newblock {\em J. Reine Angew. Math.}, 608:123--153, 2007.

\bibitem{BuffEpstein}
Xavier Buff and Adam Epstein.
\newblock Bifurcation measure and postcritically finite rational maps.
\newblock In {\em Complex dynamics}, pages 491--512. A K Peters, Wellesley, MA, 2009.

\bibitem{CallSilverman}
Gregory~S. Call and Joseph~H. Silverman.
\newblock Canonical heights on varieties with morphisms.
\newblock {\em Compositio Math.}, 89(2):163--205, 1993.

\bibitem{CantatDujardin}
Serge Cantat and Romain Dujardin.
\newblock Some rigidity results for polynomial automorphisms of {${\bf C}^2$}, 2024.
\newblock Forthcoming.

\bibitem{chambertloir2010heights}
Antoine Chambert-Loir.
\newblock Heights and measures on analytic spaces. a survey of recent results, and some remarks, 2010.
\newblock arXiv:1001.2517.

\bibitem{CTZ23}
Pietro Corvaja, Jacob Tsimerman, and Umberto Zannier.
\newblock Finite orbits in surfaces with a double elliptic fibration and torsion values of sections, 2023.
\newblock arXiv:2302.00859.

\bibitem{DemarcoANT}
Laura DeMarco.
\newblock Bifurcations, intersections, and heights.
\newblock {\em Algebra Number Theory}, 10(5):1031--1056, 2016.

\bibitem{DKY1}
Laura DeMarco, Holly Krieger, and Hexi Ye.
\newblock Uniform {M}anin-{M}umford for a family of genus 2 curves.
\newblock {\em Ann. of Math. (2)}, 191(3):949--1001, 2020.

\bibitem{DKY2}
Laura DeMarco, Holly Krieger, and Hexi Ye.
\newblock Common preperiodic points for quadratic polynomials.
\newblock {\em J. Mod. Dyn.}, 18:363--413, 2022.

\bibitem{DemarcoMavraki20}
Laura DeMarco and Niki~Myrto Mavraki.
\newblock Variation of canonical height and equidistribution.
\newblock {\em Amer. J. Math.}, 142(2):443--473, 2020.

\bibitem{DemarcoMavraki24}
Laura DeMarco and Niki~Myrto Mavraki.
\newblock Dynamics on {$\Bbb{P}^1$}: preperiodic points and pairwise stability.
\newblock {\em Compos. Math.}, 160(2):356--387, 2024.

\bibitem{DGH21}
Vesselin Dimitrov, Ziyang Gao, and Philipp Habegger.
\newblock Uniformity in {M}ordell-{L}ang for curves.
\newblock {\em Ann. of Math. (2)}, 194(1):237--298, 2021.

\bibitem{DSLecture}
Tien-Cuong Dinh and Nessim Sibony.
\newblock Dynamics in several complex variables: endomorphisms of projective spaces and polynomial-like mappings.
\newblock In {\em Holomorphic dynamical systems}, volume 1998 of {\em Lecture Notes in Math.}, pages 165--294. Springer, Berlin, 2010.

\bibitem{DujardinFavre08}
Romain Dujardin and Charles Favre.
\newblock Distribution of rational maps with a preperiodic critical point.
\newblock {\em Amer. J. Math.}, 130(4):979--1032, 2008.

\bibitem{DFManinMumford2017}
Romain Dujardin and Charles Favre.
\newblock The dynamical manin--mumford problem for plane polynomial automorphisms.
\newblock {\em Journal of the European Mathematical Society}, 19(11):3421--3465, 2017.

\bibitem{FavreGauthier2022}
Charles Favre and Thomas Gauthier.
\newblock {\em The arithmetic of polynomial dynamical pairs}, volume 214 of {\em Annals of Mathematics Studies}.
\newblock Princeton University Press, Princeton, NJ, [2022] \copyright 2022.

\bibitem{Okainequality}
John~Erik Forn\ae~ss and Nessim Sibony.
\newblock Oka's inequality for currents and applications.
\newblock {\em Math. Ann.}, 301(3):399--419, 1995.

\bibitem{FriedlandMilnor}
Shmuel Friedland and John Milnor.
\newblock Dynamical properties of plane polynomial automorphisms.
\newblock {\em Ergodic Theory Dynam. Systems}, 9(1):67--99, 1989.

\bibitem{GGK21}
Ziyang Gao, Tangli Ge, and Lars Kühne.
\newblock The uniform {M}ordell-{L}ang conjecture, 2021.
\newblock arXiv:2105.15085.

\bibitem{GaoHabeggerManinMumford}
Ziyang Gao and Philipp Habegger.
\newblock The relative manin-mumford conjecture, 2023.
\newblock arXiv:2303.05045.

\bibitem{gauthier2023good}
Thomas Gauthier.
\newblock Good height functions on quasi-projective varieties: equidistribution and applications in dynamics, 2023.
\newblock arXiv:2105.02479.

\bibitem{gauthier2023sparsity}
Thomas Gauthier, Johan Taflin, and Gabriel Vigny.
\newblock Sparsity of postcritically finite maps of $\mathbb{P}^k$ and beyond: A complex analytic approach, 2023.
\newblock arXiv:2305.02246.

\bibitem{gauthier2020geometric}
Thomas Gauthier and Gabriel Vigny.
\newblock The geometric dynamical {N}orthcott and {B}ogomolov properties, 2020.
\newblock arXiv:1912.07907.

\bibitem{GVhenon}
Thomas Gauthier and Gabriel Vigny.
\newblock The geometric dynamical {N}orthcott property for regular polynomial automorphisms of the affine plane.
\newblock {\em Bull. Soc. Math. France}, 150(4):677--698, 2022.

\bibitem{Habegger13}
Philipp Habegger.
\newblock Torsion points on elliptic curves in {W}eierstrass form.
\newblock {\em Ann. Sc. Norm. Super. Pisa Cl. Sci. (5)}, 12(3):687--715, 2013.

\bibitem{hsia2018heights}
Liang-Chung Hsia and Shu Kawaguchi.
\newblock Heights and periodic points for one-parameter families of {H}\'enon maps, 2018.
\newblock arXiv:1810.03841.

\bibitem{Hubbard86}
John~H. Hubbard.
\newblock The {H}\'{e}non mapping in the complex domain.
\newblock In {\em Chaotic dynamics and fractals ({A}tlanta, {G}a., 1985)}, volume~2 of {\em Notes Rep. Math. Sci. Engrg.}, pages 101--111. Academic Press, Orlando, FL, 1986.

\bibitem{Ingram2011CanonicalHF}
Patrick Ingram.
\newblock Canonical heights for {H}{\'e}non maps.
\newblock {\em Proceedings of the London Mathematical Society}, 108, 2011.

\bibitem{irokawa2023hybrid}
Reimi Irokawa.
\newblock Hybrid dynamics of {H}\'enon mappings, 2023.
\newblock arXiv:2212.10851.

\bibitem{daoforcurves}
Zhuchao Ji and Junyi Xie.
\newblock Dao for curves, 2023.
\newblock arXiv:2302.02583.

\bibitem{KawaguchiAffine}
Shu Kawaguchi.
\newblock Canonical height functions for affine plane automorphisms.
\newblock {\em Mathematische Annalen}, 335:285--310, 06 2006.

\bibitem{Kawaguchi13}
Shu Kawaguchi.
\newblock Local and global canonical height functions for affine space regular automorphisms.
\newblock {\em Algebra Number Theory}, 7(5):1225--1252, 2013.

\bibitem{KuhneRelativeBogomolov}
Lars K\"{u}hne.
\newblock The relative {B}ogomolov conjecture for fibered products of elliptic curves.
\newblock {\em J. Reine Angew. Math.}, 795:243--270, 2023.

\bibitem{kuhne2021equidistribution}
Lars Kühne.
\newblock Equidistribution in families of abelian varieties and uniformity, 2021.
\newblock arXiv:2101.10272.

\bibitem{Lazarsfeld}
Robert Lazarsfeld.
\newblock {\em Positivity in algebraic geometry. {I}}, volume~48 of {\em Ergebnisse der Mathematik und ihrer Grenzgebiete. 3. Folge. A Series of Modern Surveys in Mathematics [Results in Mathematics and Related Areas. 3rd Series. A Series of Modern Surveys in Mathematics]}.
\newblock Springer-Verlag, Berlin, 2004.
\newblock Classical setting: line bundles and linear series.

\bibitem{Lee2013}
Chong~Gyu Lee.
\newblock The equidistribution of small points for strongly regular pairs of polynomial maps.
\newblock {\em Math. Z.}, 275(3-4):1047--1072, 2013.

\bibitem{MZ10}
D.~Masser and U.~Zannier.
\newblock Torsion anomalous points and families of elliptic curves.
\newblock {\em Amer. J. Math.}, 132(6):1677--1691, 2010.

\bibitem{MZ15}
David Masser and Umberto Zannier.
\newblock Torsion points on families of simple abelian surfaces and {P}ell's equation over polynomial rings.
\newblock {\em J. Eur. Math. Soc. (JEMS)}, 17(9):2379--2416, 2015.
\newblock With an appendix by E. V. Flynn.

\bibitem{MavrakiSchmidt}
Niki~Myrto Mavraki and Harry Schmidt.
\newblock On the dynamical bogomolov conjecture for families of split rational maps, 2022.
\newblock arXiv:2201.10455.

\bibitem{MoriwakiRdivisor}
Atsushi Moriwaki.
\newblock Adelic divisors on arithmetic varieties.
\newblock {\em Mem. Amer. Math. Soc.}, 242(1144):v+122, 2016.

\bibitem{Pink05}
Richard Pink.
\newblock A common generalization of the conjectures of {A}ndr\'e-{O}ort, {M}anin-{M}umford, and {M}ordell-{L}ang, 2005.
\newblock Preprint.

\bibitem{qu2023arithmetic}
Binggang Qu and Hang Yin.
\newblock Arithmetic {D}emailly approximation theorem, 2023.
\newblock arXiv:2208.13230.

\bibitem{SibonyPanorama}
Nessim Sibony.
\newblock Dynamique des applications rationnelles de {$\bold P^k$}.
\newblock In {\em Dynamique et g\'{e}om\'{e}trie complexes ({L}yon, 1997)}, volume~8 of {\em Panor. Synth\`eses}, pages ix--x, xi--xii, 97--185. Soc. Math. France, Paris, 1999.

\bibitem{Stoll}
Michael Stoll.
\newblock Simultaneous torsion in the {L}egendre family.
\newblock {\em Exp. Math.}, 26(4):446--459, 2017.

\bibitem{SUZ97}
L.~Szpiro, E.~Ullmo, and S.~Zhang.
\newblock \'{E}quir\'{e}partition des petits points.
\newblock {\em Invent. Math.}, 127(2):337--347, 1997.

\bibitem{Tanlei}
Lei Tan.
\newblock Similarity between the {M}andelbrot set and {J}ulia sets.
\newblock {\em Comm. Math. Phys.}, 134(3):587--617, 1990.

\bibitem{Ullmo98}
Emmanuel Ullmo.
\newblock Positivit\'{e} et discr\'{e}tion des points alg\'{e}briques des courbes.
\newblock {\em Ann. of Math. (2)}, 147(1):167--179, 1998.

\bibitem{ICCMYuan}
Xinyi Yuan.
\newblock Algebraic dynamics, canonical heights and {A}rakelov geometry.
\newblock In {\em Fifth {I}nternational {C}ongress of {C}hinese {M}athematicians. {P}art 1, 2}, volume 51, pt. 1, 2 of {\em AMS/IP Stud. Adv. Math.}, pages 893--929. Amer. Math. Soc., Providence, RI, 2012.

\bibitem{yuan2023big}
Xinyi Yuan.
\newblock Arithmetic bigness and a uniform bogomolov-type result, 2023.
\newblock arXiv:2108.05625.

\bibitem{YZadelic}
Xinyi Yuan and Shou-Wu Zhang.
\newblock Adelic line bundles on quasi-projective varieties, 2024.
\newblock arXiv:2105.13587.

\bibitem{Zannier12}
Umberto Zannier.
\newblock {\em Some problems of unlikely intersections in arithmetic and geometry}, volume 181 of {\em Annals of Mathematics Studies}.
\newblock Princeton University Press, Princeton, NJ, 2012.
\newblock With appendixes by David Masser.

\bibitem{ZhangEquidistribution98}
Shou-Wu Zhang.
\newblock Equidistribution of small points on abelian varieties.
\newblock {\em Ann. of Math. (2)}, 147(1):159--165, 1998.

\bibitem{ZhangICM}
Shou-Wu Zhang.
\newblock Small points and {A}rakelov theory.
\newblock In {\em Proceedings of the {I}nternational {C}ongress of {M}athematicians, {V}ol. {II} ({B}erlin, 1998)}, pages 217--225, 1998.

\bibitem{ZHANGjams}
Shouwu Zhang.
\newblock Positive line bundles on arithmetic varieties.
\newblock {\em J. Amer. Math. Soc.}, 8(1):187--221, 1995.

\bibitem{zhangJAG1995}
Shouwu Zhang.
\newblock Small points and adelic metrics.
\newblock {\em J. Algebraic Geom.}, 4(2):281--300, 1995.

\end{thebibliography}

\end{document}